\documentclass{amsart}
\usepackage{miyamath,amsthm,main}
\title[Monomial Deformations and Hypergeometric Functions]{Monomial Deformations of Certain Hypersurfaces and Two Hypergeometric Functions.}
\author{Kazuaki Miyatani}
\address{National Center for Theoretical Sciences, Math. Division,\\
National Taiwan University, No. 1, Sec. 4, Roosevelt Road, Taipei}
\email{miyatani@math.i3vi3.xyz}
\date{\today}

\begin{document}
\maketitle

\begin{abstract}
    The purpose of this article is to give an explicit description, in terms of hypergeometric functions
    over finite fields, of zeta function of a certain type of smooth hypersurfaces that generalizes Dwork family.
    The point here is that we count the number of rational points
    employing both character sums and the theory of weights, which enables us to enlighten the calculation of the zeta function.
\end{abstract}

\setcounter{section}{-1}
\section{Introduction.}

Hypergeometric functions with specific parameters appear as periods of families of algebraic varieties over complex numbers.
The most classical example is the function $\ghf{2}{1}{1/2, 1/2}{1}{\lambda}$ 
apparing in the period of Legendre family of elliptic curves.
Another example is the function $\ghf{n}{n-1}{1/(n+1),2/(n+1),\dots,n/(n+1)}{1,\dots,1}{\lambda^{-(n+1)}}$
observed by Dwork \cite{Dwork:PC}.
This appears in the periods of the family of Calabi--Yau varieties defined by $T_1^{n+1}+\dots+T_{n+1}^{n+1}-\lambda (n+1)T_1\dots T_{n+1}$ in the $n$-dimensional projective spaces; this family is now called ``Dwork family''.

Turning to algebraic varieties over finite fields, hypergeometric series is still related to them.
Remarkably, a criterion for existence of unit root (a Frobenius eigenvalue on the \'etale cohomology which is a $p$-adic unit) of Legendre family, 
and the unit root itself if exists, can be described by using the hypergeometric series
$\ghf{2}{1}{1/2,1/2}{1}{x}$ considered as a formal power series with $p$-adic coefficients.
We still have an analogous theorem on Dwork families \cite{Yu:VURDFCYV}.

\vspace{10pt}

Around 1990, Greene \cite{Greene:HFFF} and Katz \cite{Katz:ESDE} independently introduced an
$\ell$-adic version of hypergeometric functions using Gauss sums,
which is called ``hypergeometric function over finite fields'' or ``Gaussian hypergeometric function''.
This function has a connection with algebraic varieties over finite fields
analogously to the classical hypergeometric functions.
For example, the Frobenius traces of first \'etale cohomology of Legendre families over $\bF_q$ (with odd $q$)
can be described by $\ghf{2}{1}{\varphi_2,\varphi_2}{\varepsilon}{\lambda}_{\bF_q}$,
where $\varphi_2$ \resp{$\varepsilon$} is a $\overline{\bQ_{\ell}}$-valued character of order $2$ \resp{trivial character} on $\bF_q^{\times}$.
We also have a similar theorem on Dwork families, which relates the zeta function of (each member of) Dwork family
over $\bF_q$ (with $q\equiv 1 \pmod{n+1}$)
with $\ghf{n}{n-1}{\varphi_{n+1},\varphi_{n+1}^2,\dots,\varphi_{n+1}^n}{\varepsilon,\dots,\varepsilon}{\lambda^{-(n+1)}}_{\bF_q}$, where $\varphi_{n+1}$ is a $\overline{\bQ_{\ell}}$-valued character of order $n+1$ on $\bF_q^{\times}$ \cite{Katz:ALDF}, \cite{Goutet:LTFZFDH}, \cite{Goutet:EFZFDH}, \cite{Goutet:IDCFZFDH}.

We have so far two following strategies for proving such a relationship between hypergeometric functions over finite fields
and the Frobenius traces of \'etale cohomology of algebraic varieties over finite fields.
The first one is counting rational points of the varieties with character sums.
This strategy is classical and employed by many mathematicians;
for example by Koike \cite{Koike:OMOHSFFECFF} for Legendre family of elliptic curves, 
by Barman and Kalita \cite{Barman-Kalita:IJNT} for more general family of curves
and by Goutet \cite{Goutet:LTFZFDH}, \cite{Goutet:EFZFDH}, \cite{Goutet:IDCFZFDH} for Dwork families.
This strategy is powerful and we may get a quite explicit result with it; we, instead,
often need a complicated calculations concerning character sums,
and it highly depends on the shape (symmetry, for example) of the defining equation of the varieties
whether the calculation is successful.
The second strategy is a geometric observation, which is remarkably employed by Katz in case of Dwork families and its generalizations \cite{Katz:ALDF};
one first decomposes the cohomology sheaf by using a finite abelian group acting on this variety,
and then relate each component with an $\ell$-adic hypergeometric sheaf.
This strategy is natural and distinguished, however, each components can be identified only modulo multiplication
by a continuous $\overline{\bQ_{\ell}}$-valued character on the absolute Galois group of $\bF_q$ \cite[Theorem 5.3, Question 5.5]{Katz:ALDF}.
\vspace{10pt}

In this article, we show more general examples of families of hypersurfaces whose zeta functions a related to
hypergeometric functions over finite fields, employing a method in which these two strategies are mixed up.
To be more precise, our method basically follows the first strategy (counting number of rational points via character sums),
but we do not forget the fact that the hypergeometric functions over finite fields has a geometric (or $\ell$-adic) nature,
and use this fact to ignore factors with improper weight.
This allows us to avoid some messy calculation of character sums, and as a result, we obtain a description of the zeta function with much more general hypersurfaces in terms of hypergeometric functions over finite fields.

The hypersurfaces on which we work is the monomial deformation of not-necessarily-diagonal hypersurfaces
of degree $n+1$ in $\bP^{n+1}_{\bF_q}$ defined by the sum of $n$ monomials,
that is, the smooth hypersurface $X_{\lambda}$, which is defined in $\bP^n_{\bF_q}$ by
\[
    F_{\lambda}(T_1,\dots,T_{n+1}) = c_1T^{a_1}+\dots+c_{n+1}T^{a_{n+1}} - \lambda T_1\dots T_{n+1},
\]
where $c_1,\dots,c_{n+1}\in\bF_q^{\times}$, $a_1,\dots,a_{n+1}\in\bZ_{\geq 0}$, and 
for $a_i=(a_{1i},\dots,a_{n+1,i})\in\bZ_{\geq 0}$, $T^{a_i}=T_1^{a_{1i}}\dots T_{n+1}^{a_{n+1,i}}$;
we assume that $X_0$ is smooth, and we view $X_{\lambda}$ as a monomial deformation of $X_0$
(moreover, we assume that $q-1$ is divisible by some numbers defined by $F_0$; see Subsection 3.1).

Finally, besides this main theme,
we note that this hypersurface $X_{\lambda}$ also has a relationship with usual hypergeometric series viewed as
$p$-adic coefficients as in the classical example.
We also refer and prove this fact in this article.

\vspace{10pt}

We conclude this introduction by explaining the structure of this article.

Section 1 is devoted to the foundation of the hypergeometric functions over finite fields.
In Subsection 1.1, we recall the definition and basic properties of hypergeometric function over finite fields,
and in Subsection 1.2, we introduce a notion of ``$q$-Weil function'' which describes the properties of 
``coming from a $\ell$-adic sheaf'' and restate the existence of $\ell$-adic hypergeometric sheaves in this terminology.

In Section 2, we start the study of the zeta function of the hypersurface $X_{\lambda}$.
Subsection 2.1 is devoted to introducing some data to describe the hypergeometric functions related to $X_{\lambda}$.
Subsection 2.2 is a preparation for Subsection 2.3, in which we describe the criterion for existence of unit root,
and even the unit root itself if exists, in terms of the usual hypergeometric series (in $p$-adic coefficients).

Section 3 is the main part of this article, which is devoted to the calculation of the zeta function of $X_{\lambda}$.
Subsection 3.1 is devoted to stating explicit assumption and introducing the other data we need for calculation.
In subsection 3.2, we state the main theorem, which describes the zeta function of $X_{\lambda}$
in terms of hypergeometric functions over finite fields;
here, we also reduce the proof to a direct calculation of character sums.
We execute the calculation in Subsection 3.3.

\subsection*{Notation.}

Throughout this article, we fix an prime number $p$, a power $q$ of $p$ and an integer $n$ greater than or equal to $2$.
The finite field with $q^r$ elements (for a positive integer $r$) is denoted by $\bF_{q^r}$.

We denote the group of characters $\bF_q^{\times}\to\obQ^{\times}$ by $\widehat{\bF_q^{\times}}$,
and the trivial character by $\epsilon$.
A character $\chi\in\widehat{\bF_q^{\times}}$ is also considered as a function $\bF_q\to\obQ$ by setting $\chi(0)=0$.
%For a positive integer $m$ that divides $q-1$, we fix a character $\varphi_m\in\widehat{\bF_q^{\times}}$ of order $m$ even if $\chi$ is trivial.

By default, elements in $A^m$ for an abelian group $A$ and a natural number $m$ are considered as column vectors.
For such a vector $b={}^{\rt}(b_1,\dots,b_m)$, we write $|b|:=b_1+\dots+b_m$.

Finally, \emph{throughout this article, we fix a non-trivial additive character $\theta\colon\bF_q\to\overline{\bQ_{\ell}}$.}
%although this choice does not affect our result, we need one anyway.

\section{Preliminaries on Hypergeometric Functions.}

\subsection{Definitions.}

In this subsection, we review the hypergeometric functions over finite fields.
First, let us recall the usual hypergeometric series.

\begin{definition}
For $2(n+1)$ rational numbers $A_1,\dots,A_{n+1},B_1,\dots,B_n$, we define the hypergeometric series by
\[
    \ghf{n+1}{n}{A_1,\dots,A_{n+1}}{B_1,\dots,B_n}{x} := \sum_{k=0}^{\infty}
    \frac{\big(A_1\big)_k\dots\big(A_{n+1}\big)_k}{\big(B_1\big)_k\dots\big(B_n\big)_k\big(1\big)_k}x^k;
\]
here, for a number $c$ and a natural number $k$, the Pochhammer symbol $\big(c\big)_k$ is defined to be
the product $c(c+1)\dots(c+k-1)$, that is, by using Gamma function $\big(c\big)_k\defeq \Gamma(c+k)/\Gamma(c)$.
\end{definition}

Now, we introduce the definition of hypergeometric function over finite fields;
recall that we fix a non-trivial additive character $\theta\colon \bF_q\to\obQ^{\times}$.
%among several existing variations, we employ the definition by McCarthy \cite[Definition 1.4]{McCarthy:TWPHFFF}
%because it seems that it fits our purpose the most.

\begin{definition}
    Let $A_1,\dots,A_{n+1}$ and $B_1,\dots,B_{n+1}$ be characters of $\bF_q^{\times}$.
    We define the \emph{hypergeometric function over $\bF_q$} with these parameters as
\[
    \tghf{n+1}{n+1}{ A_1,\dots, A_{n+1}}{B_1,\dots,B_n}{x}_{\bF_q}=
    \frac{1}{q-1} \sum_{\chi\in\widehat{\bF_q^{\times}}}\prod_{i=1}^{n+1}\frac{G\big(A_i\chi\big)}{G\big(A_i\big)}
    \prod_{i=1}^{n+1}\frac{G\big(\overline{B_i\chi}\big)}{G\big(\overline{B_i}\big)}
    \chi(-1)^{n+1}\chi(x);
\]
    here, $G\big(\mu\big)$ for a character $\mu$ is the Gauss sum with respect to the fixed $\theta$,
    that is, $G\big(\mu\big)\defeq \sum_{x\in\bF_q^{\times}}\theta(x)\mu(x)$.
    Moreover, we define
\[
    \ghf{n+1}{n}{A_1,\dots, A_{n+1}}{B_1,\dots,B_n}{x}_{\bF_q}=
    \tghf{n+1}{n+1}{A_1,\dots,A_{n+1}}{B_1,\dots,B_n,\epsilon}{x}_{\bF_q}.
\]
\end{definition}

\begin{remark}
    \label{rem:katzandmccarthy}
    In the literature, there are a lot of variants of hypergeometric functions over finite fields exist
    besides the original definition by Greene \cite{Greene:HFFF} and Katz \cite{Katz:ESDE}.
    Our function $\ghf{n+1}{n}{A_1,\dots,A_{n+1}}{B_1,\dots,B_n}{x}_{\bF_q}$ coincides with the
    definition by McCarthy \cite[Definition 1.4]{McCarthy:TWPHFFF}.
\end{remark}

    An important feature of the hypergeometric functions over finite fields is that
    it has a geometric interpretation.
    In fact, Katz \cite[8.2]{Katz:ESDE} constructs a $\overline{\bQ_{\ell}}$-sheaf
    $\sH\big(!, \theta; A_1,\dots,A_{n+1}; B_1,\dots,B_{n+1}\big)$ on $\Gm$, which is smooth on $\Gm\setminus\{1\}$
    and pure of weight $2n+1$ if $\{A_1,\dots,a_{n+1}\}\cap\{B_1,\dots,B_{n+1}\}=\emptyset$ \cite[Theorem 8.4.2 (4)]{Katz:ESDE}.
%    and its trace of Frobenius at $x$ is
%\[
%    -\sum_{(x_1,\dots,x_{n+1},y_1,\dots,y_{n+1})\in V(x)}\theta\left(\sum_{i=1}^{n+1}x_i-\sum_{i=1}^{n+1}y_i\right)A_1(x_1)\dots A_{n+1}(x_{n+1})B_1(y_1)\dots B_{n+1}(y_{n+1}),
%\]
%where
%\[
%    V(x) = \Set{ (x_1,\dots,x_{n+1},y_1,\dots,y_{n+1}) \in \big(\bF_q^{\times}\big)^{2(n+1)} |
%        x_1\dots x_{n+1}=x y_1\dots y_{n+1}}.
%\]
    We may show, in the same way as \cite[Proposition 2.6]{McCarthy:TWPHFFF}, that its trace of Frobenius at $x$ is equal to 
\[
    \prod_{i=1}^{n+1}G\big(A_i\big)G\big(B_i\big)\cdot\tghf{n+1}{n+1}{A_1,\dots,A_{n+1}}{B_1,\dots,B_{n+1}}{x^{-1}}_{\bF_q}.
\]

\subsection{$q$-Weil functions.}

In this subsection, we introduce a concept of $q$-Weil functions, and prove that the Gauss sums
and the hypergeometric functions over finite fields give examples of it.
Recall that an algebraic number $\alpha$ is said to be a \emph{$q$-Weil number of weight $k$},
where $k$ is an integer,
if for all embedding $\iota\colon\obQ\to\bC$, the absolute value of $\iota(\alpha)$ equals $q^{k/2}$.

\begin{definition}
Let $f\colon \bZ_{>0} \to \obQ$ be a function and $k$ an integer.
We say that $f$ is a \emph{pure $q$-Weil function of weight $k$} if there exist some 
$q$-Weil numbers $\alpha_1,\dots,\alpha_m$ of weight $k$ satisfying
\[
    f(r) = \sum_{i=1}^m\alpha_i^r \quad (\forall r\in\bZ_{>0}).
\]
We say that $f$ is a \emph{(mixed) $q$-Weil function} if there exist some pure $q$-Weil functions $f_i$ of weight $k_i$
$(i=1,\dots, m)$ and a number $\varepsilon_1,\dots,\varepsilon_m\in\{1, -1\}$ satisfying
\[
    f = \sum_{i=1}^m \varepsilon_if_i.
\]
$f$ is a \emph{(mixed) $q$-Weil function of weight $\leq k$} \resp{\emph{of weight $\geq k$, of weight $\neq k$}}
if all $k_i$'s can be taken to be $\leq k$ \resp{$\geq k$, $\neq k$}.
\end{definition}

The property of being $q$-Weil function can be restated in terms of the zeta function defined below.

\begin{definition}
The zeta function $\zeta(f)(T)$ of a function $f\colon\bZ_{>0}\to\obQ$ is a formal power series
\[
    \zeta(f)(T) = \exp\left(-\sum_{r=1}^{\infty}f(r)\frac{T^r}{r}\right) \in 1 + T\obQ[[T]].
\]
\label{def:zeta}
\end{definition}

The following proposition follows from a standard calculation.

\begin{proposition}
A non-zero function $f\colon \bZ_{>0}\to\obQ$ is a pure $q$-Weil function of weight $k$ if and only if
the zeta function $\zeta(f)(T)$ is a polynomial in $\overline{\bQ}[T]$ all of whose reciprocal roots are 
$q$-Weil numbers of weight $k$.
\end{proposition}

\begin{corollary}
    For each non-zero $q$-Weil function $f$, there uniquely exist mutually different numbers $k_1,\dots,k_r$ and
    non-zero $q$-Weil functions $f_i$ of weight $k_i$ for each $i$ such that $f=\sum f_i$.
\end{corollary}

To state explicitely the $q$-Weil property of Gauss sums and hypergeometric functions,
we introduce the following notation.

\begin{notation}
    Let $r$ be a positive integer.
    Then, we define the non-trivial additive character $\theta\circ\Tr_{\bF_{q^r}/\bF_q}$ on $\bF_{q^r}$ by $\theta_r$.
    Similarly, for each character $\chi$ on $\bF_q^{\times}$, the character $\chi\circ\Norm_{\bF_{q^r}/\bF_q}$ on $\bF_{q^r}^{\times}$ is denoted by $\chi_r$.
    The Gauss sum of a character $\chi'$ on $\bF_{q^r}^{\times}$ with respect to $\theta_r$
    is denoted by $G\big(\chi'\big)$.

    Moreover, for $2(n+1)$ characters $A_1,\dots,A_{n+1},B_1,\dots,B_{n+1}$ of $\bF_q^{\times}$, we define 
    $\tghf{n+1}{n+1}{A_1,\dots,A_{n+1}}{B_1,\dots,B_{n+1}}{x}_{\bF_{q^r}}$ to be
    \[
        \frac{1}{q^r-1} \sum_{\chi\in\widehat{\bF_{q^r}^{\times}}}\prod_{i=1}^{n+1}\frac{G\big(A_{i,r}\chi\big)}{G\big(A_{i,r}\big)}
        \prod_{i=1}^n\frac{G\big(\overline{B_{i,r}\chi}\big)}{G\big(\overline{B_{i,r}}\big)}
        G\big(\overline{\chi}\big)\chi(-1)^{n+1}\chi(x);
    \]
    The function $\ghf{n+1}{n}{A_1,\dots,A_{n+1}}{B_1,\dots,B_n}{x}_{\bF_{q^r}}$ is defined as
    $\tghf{n+1}{n+1}{A_1,\dots,A_{n+1}}{B_1,\dots,B_n,\varepsilon}{x}_{\bF_{q^r}}$.
    \label{def:extended_gauss_sum}
\end{notation}

\begin{proposition}
    Let $\chi$ be a character on $\bF_q^{\times}$. Then, for all positive integer $r$,
    \[
        -G\big(\chi_r\big) = \left(-G\big(\chi\big)\right)^r.
    \]
    In particular, the function $r\mapsto -G\big(\chi_r\big)$ is a pure $q$-Weil function,
    whose weight is $1$ if $\chi$ is non-trivial and is $0$ if $\chi$ is trivial.
    \label{prop:gauss_sum_is_pure}
\end{proposition}

\begin{proof}
    The first assertion is a result of Davenport and Hasse \cite[(0.8)]{Davenport-Hasse:NKGZF}.
    The second assertion follows from the first and from the standard fact that
    $-G\big(\chi_r\big)$ itself is a $q^r$-Weil number of weight $1$ if $\chi_r$ is non-trivial
    and of weight $0$ if $\chi_r$ is trivial.
\end{proof}

Now, let us state the $q$-Weil property of the hypergeometric functions on finite fields
without common characters in upper and lower parameters.

\begin{proposition}
    Let $A_1,\dots,A_{n+1}, B_1,\dots,B_{n+1}$ be characters on $\bF_q^{\times}$,
    and assume that $\{A_1,\dots,A_{n+1}\}\cap \{B_1,\dots,B_{n+1}\}=\emptyset$.
    Let $m$ be the number of trivial characters among $A_i$'s and $B_i$'s.
    Then, for all $x\in\bF_q^{\times}\setminus\{1\}$, the function
    \[
        r\mapsto \tghf{n+1}{n+1}{A_1,\dots,A_{n+1}}{B_1,\dots,B_{n+1}}{x}_{\bF_{q^r}}
    \]
    is a pure $q$-Weil function of weight $m-1$.
    \label{prop:ghf_is_pure}
\end{proposition}

\begin{proof}
%Katz \cite[8.2]{Katz:ESDE} constructs an irreducible $\overline{\bQ_{\ell}}$-sheaf
%$\sH\big(!, \theta; A_1,\dots,A_{n+1}; B_1,\dots,B_{n+1}\big)$ on $\Gm$, which is smooth on $\Gm\setminus\{1\}$.
%This is pure of weight $2n+1$ \cite[Theorem 8.4.2 (4)]{Katz:ESDE}, and its trace of Frobenius at $x$ is
%\[
%    -\sum_{(x_1,\dots,x_{n+1},y_1,\dots,y_{n+1})\in V(x)}\theta\left(\sum_{i=1}^{n+1}x_i-\sum_{i=1}^{n+1}y_i\right)A_1(x_1)\dots A_{n+1}(x_{n+1})B_1(y_1)\dots B_{n+1}(y_{n+1}),
%\]
%where
%\[
%    V(x) = \Set{ (x_1,\dots,x_{n+1},y_1,\dots,y_{n+1}) \in \big(\bF_q^{\times}\big)^{2(n+1)} |
%        x_1\dots x_{n+1}=x y_1\dots y_{n+1}}.
%\]
%By multiplicative Fourier inversion, this can be shown (in the same way as \cite[Proposition 2.6]{McCarthy:TWPHFFF}) to equal
%\[
%    \prod_{i=1}^{n+1}G\big(A_i\big)G\big(B_i\big)\cdot\tghf{n+1}{n+1}{A_1,\dots,A_{n+1}}{B_1,\dots,B_{n+1}}{x^{-1}}_{\bF_q}.
%\]
    We use the notation in the paragraph after Remark \ref{rem:katzandmccarthy}.
    Since the base extension of the $\overline{\bQ_{\ell}}$-sheaf
    $\sH\big(!, \theta; A_1,\dots,A_{n+1}; B_1,\dots,B_{n+1}\big)$ to $\bF_{q^r}$ is isomorphic to
    $\sH\big(!, \theta_{r}; A_{1,r},\dots,A_{n+1,r}; B_{1,r},\dots,B_{n+1,r}\big)$
    \cite[(8.2.6)]{Katz:ESDE},
    the function
    \[
        r\mapsto \prod_{i=1}^{n+1}G\big(A_{i,r}\big)G\big(B_{i,r}\big)\cdot
        \tghf{n+1}{n+1}{A_1,\dots,A_{n+1}}{B_1,\dots,B_{n+1}}{x^{-1}}_{\bF_{q^r}}
    \]
    is a pure $q$-Weil function of weight $2n+1$.
    In turn, Proposition \ref{prop:gauss_sum_is_pure} shows that the function
    \[
        r\mapsto \prod_{i=1}^{n+1}G\big(A_{i,r}\big)G\big(B_{i,r}\big)
    \]
    is a pure $q$-Weil function of weight $2(n+1)-m$, which shows the proposition.
    \end{proof}

%\begin{remark}
%    Under the situation above, the hypergeometric function over $\bF_q$ equals the trace function of
%    \[
%        \sH^{\can}\big(\theta; A_1,\dots, A_{n+1}; B_1,\dots,B_{n+1}\big)
%    \]
%    in the notation of Katz \cite[Section 4]{Katz:ALDF}.
%\end{remark}

    In the classical hypergeometric functions, we may cancel common numbers in upper and lower parameters
    without any change to the functions themselves.
    For the hypergeometric series over finite fields, however, the situation is different.

\begin{definition}
    Let $A_1, \dots, A_{n+1}, B_1,\dots,B_{n+1}$ be characters on $\bF_q^{\times}$.
    By changing indices, we , without loss of generality,
    assume that $\{A_1,\dots,A_{n'+1}\}$ and $\{B_1,\dots,B_{n'+1}\}$ have an empty intersection
    and that $\{A_{n'+2},\dots, A_{n+1}\}$ and $\{B_{n'+2},\dots,B_{n+1}\}$ coincide as multi-sets.
    Then, for each positive integer $r$,
    we define the \emph{hypergeometric function with reduced parameters over $\bF_{q^r}$} by
    \[
        \trghf{A_1,\dots,A_{n+1}}{B_1,\dots,B_{n+1}}{x}_{\bF_{q^r}}\defeq
        \tghf{n'+1}{n'+1}{A_1,\dots,A_{n'+1}}{B_1,\dots,B_{n'+1}}{x}_{\bF_{q^r}}.
    \]
    \label{def:reduced_parameters}
\end{definition}

\begin{proposition}
    Let $A_1,\dots,A_{n+1}, B_1,\dots,B_{n+1}$ be characters on $\bF_q^{\times}$,
    and let $m$ be the number of trivial characters among $B_i$'s.
    Fix an element $x$ of $\bF^{\times}$.

    \textup{(i)} If all $A_i$'s are non-trivial, then the function
    \[
        r\mapsto \tghf{n+1}{n+1}{A_1,\dots,A_{n+1}}{B_1,\dots,B_{n+1}}{x}_{\bF_{q^r}}
        - \trghf{A_1,\dots,A_{n+1}}{B_1,\dots,B_{n+1}}{x}_{\bF_{q^r}}
    \]
    is a $q$-Weil function of weight $\leq m-2$.

    \textup{(ii)} If exactly one of $A_i$'s is trivial, and if at least one of $B_i$'s is trivial, then the function
    \[
        r\mapsto \tghf{n+1}{n+1}{A_1,\dots,A_{n+1}}{B_1,\dots,B_{n+1}}{x}_{\bF_{q^r}}
        - q^r\trghf{A_1,\dots,A_{n+1}}{B_1,\dots,B_{n+1}}{x}_{\bF_{q^r}}
    \]
    is a $q$-Weil function of weight $\leq m-2$.
    \label{prop:cancellation}
\end{proposition}

\begin{proof}
    To prove (i), it suffices to show that, for non-trivial characters $A_1,\dots,A_n,C$ and characters $B_1,\dots,B_n$,
    the function
    \[
        r \mapsto \tghf{n+1}{n+1}{A_1,\dots,A_n, C}{B_1,\dots,B_n,C}{x}_{\bF_{q^r}} - \tghf{n}{n}{A_1,\dots,A_n}{B_1,\dots,B_n}{x}_{\bF_{q^r}}
    \]
    is a $q$-Weil function of weight $\leq m-2$.

    By definition, $\tghf{n+1}{n+1}{A_1,\dots,A_n,C}{B_1,\dots,B_n,C}{x}_{\bF_q}$ equals
    \begin{equation}
        \label{eq:ghf_over}
        \frac{1}{q-1}\sum_{\chi\in\widehat{\bF_q^{\times}}}
        \frac{G\big(A_1\chi\big)\dots G\big(A_n\chi\big)G\big(C\chi\big)}{G\big(A_1\big)\dots G\big(A_n\big)G\big(C\big)}
        \frac{G\big(\overline{B_1\chi}\big)\dots G\big(\overline{B_{n-1}\chi}\big)G\big(\overline{C\chi}\big)}{G\big(\overline{B_1}\big)\dots G\big(\overline{B_n}\big)G\big(\overline{C}\big)}
        \chi\big((-1)^{n+1}\big)\chi(x).
    \end{equation}
    Because $G\big(\mu\big)G\big(\overline{\mu}\big)$ equals $q\mu(-1)$ if $\mu\in\widehat{\bF_q^{\times}}$ is
    non-trivial and equals $q\mu(-1)-(q-1)$ if $\mu$ is trivial,
    \[
        \frac{G\big(C\chi\big)G\big(\overline{C\chi}\big)}{G\big(C\big) G\big(\overline{C}\big)} =
        \begin{cases} \dfrac{qC\chi(-1)}{qC(-1)}=\chi(-1) & \text{if } \chi\neq C^{-1},\\[15pt]
            \dfrac{qC\chi(-1)-(q-1)}{qC(-1)} = \chi(-1)-\dfrac{q-1}{q}C(-1) & \text{if } \chi = C^{-1},
        \end{cases}
    \]
    which shows that (\ref{eq:ghf_over}) equals
    \begin{eqnarray*}
        && \frac{1}{q-1}\sum_{\chi\in\widehat{\bF_q^{\times}}}
        \frac{G\big(A_1\chi\big)\dots G\big(A_n\chi\big)}{G\big(A_1\big)\dots G\big(A_n\big)}
        \frac{G\big(\overline{B_1\chi}\big)\dots G\big(\overline{B_n\chi}\big)}{G\big(\overline{B_1}\big)\dots G\big(\overline{B_n}\big)}
        \chi\big((-1)^n\big)\chi\big(x\big)\\
        && \qquad -\frac{1}{q} \frac{G\big(A_1C^{-1}\big)\dots G\big(A_nC^{-1}\big)}{G\big(A_1\big)\dots G\big(A_n\big)}
        \frac{G\big(\overline{B_1C^{-1}}\big)\dots G\big(\overline{B_nC^{-1}}\big)}{G\big(\overline{B_1}\big)\dots G\big(\overline{B_n}\big)}
        C\big((-1)^n\big)C^{-1}\big(x\big).
    \end{eqnarray*}
    The first term of this equals $\tghf{n}{n}{A_1,\dots,A_n}{B_1,\dots,B_n}{x}_{\bF_q}$.
    By Proposition \ref{prop:gauss_sum_is_pure}, 
    the second term with $q$ replaced by $q^r$ for various $r$ gives a $q$-Weil function of weight $\leq 2n-(2n-m)-2=m-2$.

    To prove (ii), we show that, for non-trivial characters $A_1\dots,A_n$, arbitrary characters $B_1,\dots,B_n$ and the trivial character $C$, the function 
    \[
        r \mapsto \tghf{n+1}{n+1}{A_1,\dots,A_n, C}{B_1,\dots,B_n,C}{x}_{\bF_{q^r}} - q^r\tghf{n}{n}{A_1,\dots,A_n}{B_1,\dots,B_n}{x}_{\bF_{q^r}}
    \]
    is a $q$-Weil function of weight $\leq m-2$, and then it suffices to use case (i).
    This is proved in the same way by using the following equation for the trivial $C$:
    \[
        \frac{G\big(C\chi\big)G\big(\overline{C\chi}\big)}{G\big(C\big)G\big(\overline{C}\big)}=
        \begin{cases}
            q\chi(-1) & \text{if } \chi\neq\epsilon,\\
            q\chi(-1) - (q-1)C(-1) & \text{if } \chi=\epsilon.
        \end{cases}
    \]
\end{proof}

\section{Monomial Deformations and $p$-adic Hypergeometric Series.}
\subsection{Families of hypersurfaces considered.}

In this subsection, we introduce families of hypersurfaces on which we work, and set some notations concerning it.
Let $X_0$ be the \emph{smooth} hypersurface of $\bP^n_{\bF_q}$ defined by the polynomial
\[
    F_0(T) = c_1T^{a_1}+\dots+c_{n+1}T^{a_{n+1}} \in \bF_q[T_1,\dots,T_{n+1}],
\]
where $c_1,\dots,c_{n+1}\in\bF_q^{\times}$ and where $a_1,\dots,a_{n+1}\in\bN^{n+1}$ with $|a_i|=n+1
~(i=1,\dots,n+1)$, none of $a_i$'s being equal to ${}^\rt(1,1,\dots,1)$.
Here, for $a_i={}^\rt(a_{1i},\dots,a_{n+1,i})$, the notation $T^{a_i}$ denotes the monomial $T_1^{a_{1i}}\dots T_{n+1}^{a_{n+1,i}}$.

The theme of this article is investigating the monomial deformation $X_{\lambda}$ of $X_0$ defined by the polynomial
\[
    F_{\lambda}(T) = c_1T^{a_1}+\dots +c_{n+1}T^{a_{n+1}} - \lambda T_1\dots T_{n+1},
\]
where $\lambda$ moves $\bF_q^{\times}$;
mainly, we restrict our attention to $\lambda$ such that $X_{\lambda}$ is smooth.

By the following proposition, we may and do change the indices so that each diagonal entry of $A$ equals
$n+1$ or $n$.

\begin{proposition}
    Let $A$ be the matrix $A=(a_1, \dots, a_{n+1})$.
    Then, after a suitable change of indices of $a_i$'s, each diagonal entry of $A$ equals $n+1$ or $n$,
    the other entries are either $0$ or $1$, and moreover there exists at most one $1$ in each row.
\label{prop:form_of_A}
\end{proposition}

\begin{proof}
    First, let us prove the first property.
    It suffices to show that, for each index $i=1,\dots,n+1$, there exists an index $j$ with $a_{ij}\geq n$.
    Assume that there exists an $i$ that satisfies $a_{ij}\leq n-1~(j=1,\dots,n+1)$.
    Fix such an $i$, and let $P_i$ denote the point $[0:\dots:0:1:0:\dots:0]$ in $\bP_{\bF_q}^n$, where $1$ sits in the $i$-th entry.
    Then, for each $j, k\in\{1,\dots,n+1\}$,
    \[
        \frac{\partial T^{a_j}}{\partial T_k} = 
        \begin{cases} a_{kj}T^{a_j-e_k} & \text{if } a_{kj} \geq 1, \\ 0 & \text{if } a_{kj} = 0.\end{cases}
    \]
    The value of this partial derivative at $P_i$ is zero unless $T^{a_j-e_k}=T_i^n$,
    which is impossible because of the assumption on $i$.
    Therefore, we have shown that
    \[
        \frac{\partial F_0}{\partial T_k}\big(P_i\big) = 0 \quad (k=1,\dots,n+1) \quad \text{and} \quad F_0\big(P_i\big) = 0,
    \]
    and consequently $P_i$ is a singular point of $X_0$, which contradicts the hypothesis.

    The second property follows from the first and from the assumption $|a_i|=n+1$ for all $i$.

    In order to prove the third property, we assume the contrary.
    Then, after a change of coordinates, we may assume that $a_{12}=a_{13}=1$, that is, $T^{a_2}=T_1T_2^n$ and $T^{a_3}=T_1T_3^n$.
    For an element $x$ of $\overline{\bF_q}$ that satisfies $c_2+c_3x^n=0$,
    the point $P=[0:1:x:0:\dots:0]$, which is actually an $\overline{\bF_q}$-rational point of $X_0$,
    gives a singular point of $X_0$.
    In fact, the choice of $x$ shows that $\partial F_0/\partial T_1\big(P\big)=0$,
    and it is straightforward to check that $\partial F_0/\partial T_i\big(P)=0$ for all $i\geq 2$;
    thus we have derived a contradiction.
\end{proof}

In order to write the parameters and the input of the hypergeometric functions,
we introduce some notation.

\begin{proposition}
The kernel of the homomorphism $\bZ^{n+1}\to \bZ^{n+1}$ defined by the matrix $A'=\big(a_{ij}-1\big)_{i,j=1,\dots,n+1}$ is free of rank one
and generated by a vector ${}^\rt(\alpha_1,\dots, \alpha_{n+1})$ with all $\alpha_i>0$.
\label{prop:alpha}
\end{proposition}

\begin{proof}
    The matrix $A'$ is not invertible since $(1, 1,\dots, 1) A' = 0$.

    It, therefore, suffices to show that every entries of a non-zero vector ${}^\rt(x_1,\dots,x_{n+1})$
    in the kernel have the same sign.
    By multiplying every entries by $-1$ if necessary, we may assume that at least two entries are non-negative.
    After a suitable change of indices, we may assume that $x_1 \geq x_2 \geq \dots \geq x_{n+1}$;
    consequently, $x_1>0$ and $x_2\geq 0$.
    
    Now, if an index $i\in\{1,\dots,n\}$ satisfies $x_i>0$ and $x_{i+1}\leq 0$, then 
    \[
        (a_{11}-1)x_1 + \sum_{j=i+1}^{n+1} (a_{1j}-1)x_j = \sum_{j=2}^i -(a_{1j}-1)x_j.
    \]
    The left-hand side is greater than or equal to $(n-1)x_1$ since $(a_{1j}-1)x_j\geq 0$ for $j=i+1,\dots,n+1$.
    Now, $-(a_{1j}-1)$ being $0$ or $1$ for $j\neq 1$, the right-hand side is less than or equal to $(i-1)x_1$.
    This shows that $i=n$, which implies the contradiction
    \[
        0 = \sum_{j=1}^n (a_{n+1,j}-1)x_j + (a_{n+1,n+1}-1)x_{n+1}<0
    \]
    because $a_{n+1,j}=0$ for at least one $j\in\{1,\dots,n\}$.
\end{proof}

We freely use the notation $\alpha_1,\dots,\alpha_{n+1}$, whose choice is obviously unique, throughout this article.
The sum $\sum_{i=1}^{n+1} \alpha_i$ is denoted by $\alpha$.
Throughout this section, we always assume the following condition holds.

\begin{assumption}
    $q$ is relatively prime to all $\alpha_i$'s and to $\alpha$.
    \label{asm}
\end{assumption}

\begin{definition}
    We define an element $C\in\bF_q^{\times}$ by
    \[
        C = \alpha^{\alpha}\frac{c_1^{\alpha_1}}{\alpha_1^{\alpha_1}}\dots \frac{c_{n+1}^{\alpha_{n+1}}}{\alpha_{n+1}^{\alpha_{n+1}}}.
    \]
    \label{def:C}
\end{definition}

\begin{example}
    (i) In the case of the Dwork family, that is, if $F_0(T)=T_1^{n+1}+\dots+T_{n+1}^{n+1}$,
    then ${}^{\rt}(\alpha_1,\alpha_2,\dots,\alpha_{n+1}) = {}^{\rt}(1, 1, \dots, 1)$.
    Since $\alpha=n+1$, we have $C=(n+1)^{n+1}$.

    (ii) Let us consider the following example discussed by Yu and Yui \cite[(4.8.1)]{Yu-Yui:KSFHFF}:
    $n=3$ and $F_0(T)=T_1^4+T_1T_2^3+T_3^4+T_4^4$.
    Then, we have ${}^{\rt}(\alpha_1, \alpha_2,\alpha_3,\alpha_4) = {}^{\rt}(2, 4, 3, 3)$.
    Since $\alpha=12$, we have $C=2^{14}\cdot 3^6$.
    \label{ex:section1_example}
\end{example}

\subsection{Formal group law.}

In this subsection, we recall some facts on formal group laws and a special case of Artin--Mazur functors,
and apply the theory to our case.
For the basic language of formal group laws, the reader may consult Hazewinkel's book \cite{Hazewinkel:FGA}.

Let $R$ be a (commutative unitary) ring.
A one-dimensional commutative formal group law over $R$, which we simply say ``a formal group law over $R$" in this article,
is a formal power series $G(X,Y)\in R[[X,Y]]$ that satisfies some conditions corresponding to group axioms
\cite[1.1]{Hazewinkel:FGA}.
A logarithm of the formal group law $G$ is a formal power series $l(\tau)\in R[[\tau]]$ that satisfies
$l(\tau)\equiv\tau\pmod{\deg\,\,2}$ and $G(X, Y) = l^{-1}\big(l(X) + l(Y)\big)$.
Equivalently, a logarithm of $G$ is a strict isomorphism of $G$ to the additive group $\widehat{\bG_\ra}$
defined by $\widehat{\bG_\ra}(X, Y) = X + Y$.
A logarithm of a formal group law is unique if the ring $R$ is of characteristic zero, and
it exists if $R$ contains $\bQ$.
If $R$ is of characteristic zero, we also say that $l(\tau)\in R\otimes_\bZ\bQ[[\tau]]$ is a logarithm of
the formal group $G(X, Y)$ over $R$, if it is the logarithm of $G(X, Y)$ viewed as a formal group law
over $R\otimes_\bZ\bQ$.

For a formal group law $G(X, Y)$ over a ring $R$ and a positive integer $m$,
define a formal power series $[m]_G(X)$ inductively by
\[
    [1]_G(X) = X,\qquad [m]_G(X)=G\big(X, [m-1]_G(X)\big)\quad (m\geq 2).
\]
Assume that $R$ is a field of characteristic $p$.
%Now, let $G(X, Y)$ be a formal group law over a field $k$ of characteristic $p$.
If $[p]_G(X)$ is non-zero, then the lowest term of $[p]_G(X)$ is of degree $p^h$ for some positive integer $h$.
We call this $h$ the \emph{height} of $G(X, Y)$; if $[p]_G(X)$ is zero, the height is defined to be infinity.

%The lemma below is often used to tell whether a formal group law over $k$ is of height one.

\begin{lemma}
    Let $R$ be a ring of characteristic zero whose reduction modulo $p$ is a field $k$.
    Let $G(X, Y)$ be a formal group law over $R$ with logarithm of the form
    $l(\tau)=\sum_{s=0}^{\infty}a_s\tau^{p^s}/p^s~(a_s\in R,~a_0=1)$.
    Then, the formal group law $\overline{G}(X, Y)\defeq G(X, Y) \mod ~p$ over $k$ is of height one if and only if $a_1\not\equiv 0 \pmod{p}$.
    \label{lem:height1}
\end{lemma}

\begin{proof}
    Let $u$ denote the coefficient of $X^p$ in the formal power series $[p]_G(X)$.
    Then, $[p]_G(X)$ is of the form
    \begin{equation}
        [p]_G(X) = pX + v_2X^2+\dots+v_{p-1}X^{p-1}+uX^p + \big(\text{degree }\geq p+1\big)
        \label{eq:multiplication_p}
    \end{equation}
    where $v_2,\dots,v_{p-1}\equiv 0 \pmod{p}$;
    the formal group law $\overline{G}$ is of height one if and only if $u\not\equiv 0\pmod{p}$.
    The definition of the logarithm implies the equation
    \[
        l\big([p]_G(X)\big) = pl\big(X\big) = pX + a_1X^p + a_2\frac{X^{p^2}}{p} + \dots.
    \]
    By substituting (\ref{eq:multiplication_p}), the coefficient of $X^p$ in the left-hand side equals $u$ modulo $p$,
    and the lemma follows.
\end{proof}

%Before explaining the Artin--Mazur functor \cite{Artin-Mazur:FGAV},
%let us recall the formal group.

Next, we briefly recall the theory of formal groups.
Let $R$ be a ring and let $\NilAlg_R$ denote the category of nilpotent $R$-algebras.
For a natural number $n$, the functor $\widehat{\bA}^n\colon\NilAlg_R\to\cSet$ is defined to be $\widehat{\bA}^n(N)=N^n$. We say that a functor $G\colon\NilAlg_R\to\Ab$ is an \emph{$n$-dimensional formal group}
if its composition with the forgetful functor $\Ab\to\cSet$ is isomorphic to $\widehat{\bA}^n$.
The category of (one-dimensional commutative) formal group laws and the category of one-dimensional formal groups are equivalent.

Let $X$ be a scheme over a a perfect field $k$ of characteristic $p$, and $i$ a natural number.
Then, we define the {\em Artin--Mazur functor} by
\[
    H^i\big(X, \widehat{\Gm}\big)\colon \NilAlg_R\to \Ab;\quad A\mapsto H^i\big(X, \widehat{\Gm}(\sO_X\otimes_R A)\big).
\]
Although this may or may not be a formal group, 
it actually is a formal group in the case of our interest.
In fact, the Artin--Mazur functor is deeply related to the Witt cohomology of a variety over $k$;
if $X$ be a complete intersection of dimension $d \geq 2$,
then $H^d\big(X, \widehat{\Gm}\big)$ is a formal group,
and the Cartier--Dieudonn\'e module of its formal group is isomorphic to $H^d\big(X, W\sO_X\big)$
\cite[II (4.2), (4.3)]{Artin-Mazur:FGAV}.

Moreover, assume that $k=\bF_q$ and that $X$ is (complete intersection of degree $d\geq 2$ and)
proper and smooth over $k$. 
Then, recall that the Witt cohomology is, after the base extension to $K:=\Frac(W(k))$, isomorphic to 
the maximal subspace $H^d_{\cris}\big(X/W(k)\big)_{K}^{<1}$ of the crystalline cohomology
$H^d_{\cris}\big(X/W(k)\big)_K$ on which the Frobenius acts with slope $<1$.
If the height $h$ of the formal group law $H^d\big(X, \widehat{\Gm}\big) \bmod p$ is finite,
then, $h$ is equal to the dimension of $H^d_{\cris}\big(X/W(k)\big)_K^{<1}$ since, in fact, both equal the rank of
Cartier--Dieudonn\'e module associated to the latter \cite[II, Remarques 2.15 (a)]{Illusie:CRWCC},
\cite[(A.13)]{Stienstra-Beukers:PFEFBGCEKS}.
Because the slope of Frobenius equals $(h-1)/h$,
the first slope of Newton polygon of the crystalline cohomology is zero if and only if $h=1$ \cite{Demazure:LPG}.
In this case, the unique eigenvalue of the Frobenius on $H^d_{\cris}\big(X/W(k)\big)_K$ that is a $p$-adic unit
is called unit-root of $X_{\lambda}$.

%Stienstra \cite[Theorem 1]{Stienstra:FGLAV} gives an explicit description of the logarithm of Artin--Mazur functor
%under some conditions.
%We may apply his theorem to the study of our hypersurfaces
%the following proposition is a very special case of this theorem
%(note that we have shifted the index of $\beta_m$ by one).
%
%\begin{proposition}
%    Let $R$ be a noetherian ring which is flat over $\bZ$, and
%    let $X$ be the flat hypersurface of $\bP_R^N$ defined by a homogeneous polynomial 
%    $F(T)\in R[T_1,\dots,T_{N+1}]$ of degree $N+1$.
%    Then, $H^{N-1}\big(X, \widehat{\Gm}\big)$ is a formal group
%    whose logarithm $l(\tau)$ is given by
%    \[
%        l(\tau) = \sum_{m=0}^{\infty}\beta_m\frac{\tau^{m+1}}{m+1},
%    \]
%    where the $\beta_m$ is the coefficient of $T_1^m\dots T_{N+1}^m$ of $F(T)^m$.
%    \label{prop:S_theorem1}
%\end{proposition}

\begin{theorem}
    Let $R$ be a flat $W(\bF_q)$-algebra, and let $\Lambda$ be an element of $R$.
    Let $\Lambda$ be an element of $R$ such that $\widetilde{X_{\Lambda}}$ is flat over $R$.
    Let $\widetilde{X_{\Lambda}}$ be the hypersurface of $\bP^n_{R}$ defined by
    \[
        \widetilde{F_{\Lambda}}(T) \defeq \widetilde{c_1}T^{a_1}+\dots+\widetilde{c_{n+1}}T^{a_{n+1}}-\Lambda T_1\dots T_{n+1} \in R[T_1,\dots,T_{n+1}],
    \]
    and assume that $\widetilde{X_{\Lambda}}$ is flat over $R$.
    Then, the Artin--Mazur functor $H^{n-1}(\widetilde{X_{\Lambda}}, \widehat{\Gm})$ is a formal group
    the logarithm $l(\tau)$ of whose formal group law is given by
    \[
        \sum_{m=0}^{\infty}(-\Lambda)^m\ghf{\alpha}{\alpha-1}{\frac{-m}{\alpha},\frac{-m+1}{\alpha},\dots,\frac{-m+\alpha-1}{\alpha}}
        {\frac{1}{\alpha_1},\dots,\frac{\alpha_1-1}{\alpha_1},1,\frac{1}{\alpha_2},\dots,\frac{\alpha_2-1}{\alpha_2},1,\dots,\frac{1}{\alpha_{n+1}},\dots,\frac{\alpha_{n+1}-1}{\alpha_{n+1}}}{\widetilde{C}\Lambda^{-\alpha}}\frac{\tau^{m+1}}{m+1}.
    \]
    \label{prop:fgl_formula}
\end{theorem}

\begin{proof}
    We already know that $H^{n-1}(\widetilde{X_{\Lambda}}, \widehat{\Gm})$ is a formal group.
    The theorem of Stienstra \cite[Theorem 1]{Stienstra:FGLAV} shows that
    the coefficient of $\tau^{m+1}/(m+1)$ of its logarithm equals
    the coefficient of $T_1^m\dots T_{n+1}^m$ of $\widetilde{F_{\Lambda}}^m$.
    
    Now, let us calculate the coefficient of the latter.
    Looking at the binomial expansion

%\begin{lemma}
%    Let $\Lambda$ be an element of $R$.
%    Then, for each positive integer $m$, the coefficient of $T_1^m\dots T_{n+1}^m$ of $\widetilde{F_{\Lambda}}(T)^m$ equals
%\begin{equation}
%    (-\Lambda)^m\ghf{\alpha}{\alpha-1}{\frac{-m}{\alpha},\frac{-m+1}{\alpha},\dots,\frac{-m+\alpha-1}{\alpha}}
%    {\frac{1}{\alpha_1},\dots,\frac{\alpha_1-1}{\alpha_1},1,\frac{1}{\alpha_2},\dots,\frac{\alpha_2-1}{\alpha_2},1,\dots,\frac{1}{\alpha_{n+1}},\dots,\frac{\alpha_{n+1}-1}{\alpha_{n+1}}}{\widetilde{C}\Lambda^{-\alpha}}.
%    \label{eq:coefficient_ghf}
%\end{equation}
%{\rm (}Note that, in the lower parameters of ${}_{\alpha}F_{\alpha-1}$, the last $1$ is dropped.
%The element $C$ is the one defined in Definition \ref{def:C}.{\rm )}
%\label{lem:coefficient_ghf}
%\end{lemma}

%\begin{proof}
    \[
        \widetilde{F_{\Lambda}}(T)^m = \sum_{m_1+\dots+m_{n+1}+m'=m}\frac{m!}{m_1!\dots m_{n+1}!m'!}\prod_{i=1}^{n+1}\widetilde{c_i}^{m_i}T^{m_ia_i}(-\Lambda T_1\dots T_{n+1})^{m'},
    \]
    we notice that the coefficient of $T_1^m\dots T_{n+1}^m$ is the sum
    \[
        \sum_{(m_1,\dots,m_{n+1},m')}\frac{m!}{m_1!\dots m_{n+1}!m'!}\cdot \widetilde{c_1}^{m_1}\dots \widetilde{c_{n+1}}^{m_{n+1}}(-\Lambda)^{m'},
    \]
    where the index runs so that
    \begin{eqnarray*}
        m_1a_{i1} + \dots + m_{n+1}a_{i,n+1}+m' &=& m \quad (i=1,\dots, n+1),\\
        m_1 + \dots + m_{n+1} + m' &=& m.
    \end{eqnarray*}
    By Proposition \ref{prop:alpha}, an $(n+2)$-tuple $(m_1,\dots,m_{n+1},m)$ satisfies this condition if and only if the equation
    \[
        (m_1,\dots,m_{n+1}) = k(\alpha_1,\dots,\alpha_{n+1}),\quad m' = m-k\alpha
    \]
    holds for a natural number $k$. This shows that the coefficient equals
    \begin{equation}
        \sum_{k\geq 0, m\geq k\alpha}\frac{m!}{(k\alpha_1)!\dots (k\alpha_{n+1})!(m-k\alpha)!}\widetilde{c_1}^{k\alpha_1}\dots \widetilde{c_{n+1}}^{k\alpha_{n+1}}(-\Lambda)^{m-k\alpha}.
        \label{eq:coefficient_temp}
    \end{equation}

    On the other hand, the number
    \begin{equation}
        (-\Lambda)^m\ghf{\alpha}{\alpha-1}{\frac{-m}{\alpha},\frac{-m+1}{\alpha},\dots,\frac{-m+\alpha-1}{\alpha}}
        {\frac{1}{\alpha_1},\dots,\frac{\alpha_1-1}{\alpha_1},1,\frac{1}{\alpha_2},\dots,\frac{\alpha_2-1}{\alpha_2},1,\dots,\frac{1}{\alpha_{n+1}},\dots,\frac{\alpha_{n+1}-1}{\alpha_{n+1}}}{\widetilde{C}\Lambda^{-\alpha}}.
        \label{eq:coefficient_ghf}
    \end{equation}
    equals, by definition of hypergeometric function,
    \begin{eqnarray*}
        && (-\Lambda)^m\sum_{k=0}^{\infty}\frac{\big(\frac{-m}{\alpha}\big)_k\big(\frac{-m+1}{\alpha}\big)_k\dots\big(\frac{-m+\alpha-1}{\alpha}\big)_k}{\big(\frac{1}{\alpha_1}\big)_k\dots\big(\frac{\alpha_1}{\alpha_1}\big)_k\dots\big(\frac{1}{\alpha_{n+1}}\big)_k\dots\big(\frac{\alpha_{n+1}}{\alpha_{n+1}}\big)_k}
       \left(\frac{\alpha^{\alpha}}{\Lambda^{\alpha}}\frac{\widetilde{c_1}^{\alpha_1}}{\alpha_1^{\alpha_1}}\dots\frac{\widetilde{c_{n+1}}^{\alpha_{n+1}}}{\alpha_{n+1}^{\alpha_{n+1}}}\right)^k \\
       &=& \sum_{k=0}^{\infty}\frac{\alpha^{k\alpha}\big(\frac{-m}{\alpha}\big)_k\big(\frac{-m+1}{\alpha}\big)_k\dots\big(\frac{-m+k\alpha-1}{\alpha}\big)_k}{A_{1,k}\dots A_{n+1,k}} \widetilde{C'_k}(-\Lambda)^{m-k\alpha},
    \end{eqnarray*}
    where
    \[
        A_{i,k} = \alpha_i^{k\alpha_i}\big(\tfrac{1}{\alpha_i}\big)_k\dots\big(\tfrac{\alpha_i-1}{\alpha_i}\big)_k\big(\tfrac{\alpha_i}{\alpha_i}\big)_k
        \quad (i=1,\dots,n+1)
    \]
    and $C'_k = c_1^{k\alpha_1}\dots c_{n+1}^{k\alpha_{n+1}}$.
    Moreover, each summand is $0$ if $k-m/\alpha\geq 0$.
    We, in fact, have $A_{i,k}=(k\alpha_i)!$ and
    \begin{eqnarray*}
        && (-1)^{k\alpha}\alpha^{k\alpha}\big(\tfrac{-m}{\alpha}\big)_k\big(\tfrac{-m+1}{\alpha}\big)_k\dots\big(\tfrac{-m+k\alpha-1}{\alpha}\big)_k \\
        &=& (-1)^{k\alpha} (-m)(-m+1)\dots (-m+k\alpha-1) = \frac{m!}{(m-k\alpha)!}.
    \end{eqnarray*}
    This shows that two numbers (\ref{eq:coefficient_temp}) and (\ref{eq:coefficient_ghf}) coincide with each other.
\end{proof}

\subsection{Unit root.}

Let $\sF(x)$ denote the formal power series
\begin{equation}
    \ghf{\alpha}{\alpha-1}{\frac{1}{\alpha},\frac{2}{\alpha},\dots,\frac{\alpha-1}{\alpha},1}
    {\frac{1}{\alpha_1},\dots,\frac{\alpha_1-1}{\alpha_1},1,\frac{1}{\alpha_2},\dots,\frac{\alpha_2-1}{\alpha_2},1,\dots,\frac{1}{\alpha_{n+1}},\dots,\frac{\alpha_{n+1}-1}{\alpha_{n+1}}}{\widetilde{C}x} \in W\big(\bF_q\big)[[x]].
    \label{eq:padic_ghf}
\end{equation}
In this subsection, by a method of Stienstra and Beukers \cite{Stienstra-Beukers:PFEFBGCEKS},
we prove that $\sF(x)$ gives us information on unit root of $X_{\lambda}$.
(Our argument is a generalization of the argument by Yu \cite[Section 5]{Yu:VURDFCYV}.)

For positive integers $m$ and $s$, let $\sF_{m,s}(x)$ denote the polynomial obtained by truncating $\sF(x)$ 
up to degree $mp^s-1$.
Let $\sR$ denote the $p$-adic completion of the ring
\[
    W\big(\bF_q\big)\big[x, \big(x\sF_{1,1}(x)\big)^{-1}\big],
\]
to which the Frobenius endomorphism $\sigma$ on $W\big(\bF_q\big)$ extends by $\sigma(x)=x^p$.

\begin{lemma}
    The formal power series
    \[
        f(x) \defeq \frac{\sF(x)}{\sigma\big(\sF(x)\big)} \in W\big(\bF_q\big)[[x]]
    \]
    is actually an element of $\sR$.
    \label{lem:fisseries}
\end{lemma}

\begin{proof}
    Define a polynomial $G_{\mu,s}(x)$, for positive integers $\mu$ and $s$, by
    \[
        G_{\mu,s}(x) = \ghf{\alpha}{\alpha-1}{\frac{-\mu p^s+1}{\alpha}, \frac{-\mu p^s+2}{\alpha}, \dots, \frac{-\mu p^s+\alpha}{\alpha}}{\frac{1}{\alpha_1},\dots,\frac{\alpha_1-1}{\alpha_1},1,\frac{1}{\alpha_2},\dots,\frac{\alpha_2-1}{\alpha_2},1,\dots,\frac{1}{\alpha_{n+1}},\dots,\frac{\alpha_{n+1}-1}{\alpha_{n+1}}}{\widetilde{C}x}.
    \]
    Let $G'_{\mu,s}(t)$ be the polynomial defined by
    \[
        G'_{\mu,s}(t) = (-t)^{\mu p^s-1} G_{\mu,s}(t^{-\alpha}).
    \]
    This is the coefficient of $\tau^{\mu p^s}/\mu p^s$ in the logarithm $l(\tau)$ in Theorem \ref{prop:fgl_formula},
    applied with $R$ being the $p$-adic completion $\sS$ of the ring
    \[
        W\big(\bF_q\big)\left[t, \big(t\sF_{1,1}(t^{-\alpha})\big)^{-1}\right]
    \]
    and with $\Lambda=t$.
    Now, since we have $G'_{1,1}(t) \equiv t^{p-1}\sF_{1,1}(t^{-\alpha}) \pmod{p}$ by a straightforward observation,
    a general fact on formal group \cite[(A.8), (i)$\Longrightarrow$(v)]{Stienstra-Beukers:PFEFBGCEKS}
    shows that there exists an element $g'$ of $\sS$ independent of $\mu$ and $s$ that satisfies
    \[
        G'_{\mu,s+1}(t)\equiv g'(t)\cdot\sigma\big(G'_{\mu,s}(t)\big)\pmod{p^{s+1}}\qquad (\mu, s\geq 1).
    \]
    Thus, we have
    \[
        t^{p-1}G_{\mu,s+1}(t^{-\alpha})\equiv g'(t) \sigma\big(G_{\mu,s}(t^{-\alpha})\big) \pmod{p^{s+1}}.
    \]
    Since the polynomial $G_{\mu, s}(x)$ converges $p$-adically to $\sF(x)$ as $s\to\infty$,
    the power series $f(x)$ equals $t^{-(p-1)}g'(t)$, with $x=t^{-\alpha}$.
    This element, therefore, lies in the intersection of $W\big(\bF_q\big)[[x]]$ and the ring $\sS$,
    which equals $\sR$.
\end{proof}

\begin{theorem}
    Let $\lambda$ be an element of $\bF_q^{\times}$ such that $X_{\lambda}$ is smooth.
    Then, the first slope of the Newton polygon of $H^{n-1}_{\cris}\big(X_{\lambda}/W(\bF_q)\big)$ is zero if and only if
    $\sF_{1,1}(\lambda^{-\alpha})\neq 0$ in $\bF_q$.
    In this case, the unit root of $X_{\lambda}$ equals
    \begin{equation}
        \prod_{i=0}^{r-1}\sigma^i\big(f(\widetilde{\lambda}^{-\alpha})\big),
        \label{eq:unitroot}
    \end{equation}
    where $q=p^r$.
\label{thm:unitroot}
\end{theorem}

\begin{proof}
    Put $a \defeq \widetilde{\lambda}^{p-1}f(\widetilde{\lambda}^{-\alpha})$ where $f$ is the element in Lemma \ref{lem:fisseries}
    and put, for each $s\geq 0$, $a_s \defeq a\sigma(a)\dots\sigma^s(a)$.
    If we denote by $G'$ the formal group law over $W\big(\bF_q\big)$ with logarithm
        $l'(\tau) = \sum_{s=0}^{\infty}a_s\tau^{p^s}/p^s$,
        then, $G'$ is strictly isomorphic to the formal group law realizing $H^{n-1}\big(\widetilde{X_{\lambda}}, \widehat{\Gm}\big)$ by \cite[(A.9), (iii)$\Longrightarrow$(i)]{Stienstra-Beukers:PFEFBGCEKS}
    because $a$ satisfies the condition (iii) there.
    The first half of the theorem follows from Lemma \ref{lem:height1} because $a\equiv\sF_{1,1}(\lambda^{-\alpha}) \pmod{p}$.
    The proof of the second half goes exactly as in \cite[p.76, proof of Theorem 4.3 (2)]{Yu:VURDFCYV}.
\end{proof}

%\begin{proof}

%    Now, assume that $a$ is a $p$-adic unit.
%    Then, the argument of Stienstra--Beukers \cite[(A.13)]{Stienstra-Beukers:PFEFBGCEKS} shows that
%    the Cartier--Dieudonn\'e module of $G'$, which is rank one
%    (as is explained just after Proposition \ref{prop:relation_to_crystalline}), has a basis
%    $\omega$ on which the Frobenius acts by $\omega\mapsto a\omega$;
%    Now, the Frobenius endomorphism of the Cartier--Dieudonn\'e module of $G$ equals
%    \[
%        a\sigma(a)\dots\sigma^{r-1}(a),
%    \] which equals the number (\ref{eq:unitroot}). 
%    Now, Proposition \ref{prop:relation_to_crystalline} shows the statement.
%\end{proof}

\section{Factorization of zeta function.}
\label{section:factorization_rat}

\subsection{Preparation for stating the main theorem.}
\label{ss:precise_statement}

    In this subsection, we introduce assumptions and notations which we need to state the main theorem of this article, that is,
    the precise description of the zeta function of $X_{\lambda}$ using hypergeometric functions over finite fields.

    Let $N$ be a positive integer that divides all $\alpha_i$'s and $\alpha$, and let 
    \[
        f_N\colon \frac{\big(\bZ/N\bZ\big)^{n+1}}{{}^{\rt}(\alpha_1,\dots,\alpha_{n+1})}\to \big(\bZ/N\bZ\big)^{n+1}
    \]
    denote the morphism defined by the endomorphism of $\big(\bZ/N\bZ)^{n+1}$ defined by the matrix $\widetilde{A} \bmod{N}$. 

    Let $d_1,\dots,d_n$ be non-zero elementary divisors of $\widetilde{A}$, and put $d\defeq d_1\dots d_n$.
    Then, $\Ker\big(f_N\big)$ consists of $d$ elements by the assumption on $N$
    because $\Im\big(f_N\big)\cong\oplus_{i=1}^n d_i\bZ/N\bZ$.
    In the rest of this section, we fix $d$ vectors $s_0=0, s_1,\dots,s_{d-1}\in\bN^{n+1}$ that represent $\Ker\big(f_{q-1}\big)$ so that all entries of these vectors are in $\{0,1,\dots,q-2\}$.
    Let us write $s_j = {}^{\rt}(s_{1j},\dots,s_{n+1,j})$.

    \begin{assumption}
        For all $i$ and $j$, each $s_{ij}$ is divisible by $\alpha_i$ and $|s_j|$ is divisible by $\alpha$.
        \label{asm1}
    \end{assumption}

    Under this assumption, we put $t_{ij}\defeq s_{ij}/\alpha_i$ and $t_j\defeq |s_j|/\alpha$.

    Note that Assumption \ref{asm1} is always satisfied if we replace $\bF_q$ by an extension $\bF_{q^r}$;
    in fact, replacing $q$ by $q^r$ invokes the multiplication by $(q^r-1)/(q-1)$,
    which is an isomorphism $\Ker\big(f_{q-1}\big)\to\Ker\big(f_{q^r-1}\big)$
    (it is isomorphism because it is injective and the order of the target equals that of the source).
    Since $|s_j|<n(q-1)<\alpha(q-1)$, we have $t_j\in\{0,1,\dots,q-2\}$, and $t_j=0$ if and only if $j=0$.

\begin{assumption}
    Let $J=\{j_1,\dots,j_t\}$ be an arbitrary subset $\{1,2,\dots,n+1\}$ with $t\geq (n+1)/2$ elements,
    and denote $i_1,\dots,i_s$ all indices $i=1,\dots,n+1$ such that $a_{ij}=0$ for all $j\not\in J$.
    Then, we assume that all elementary divisors of
    \[
        \begin{pmatrix}
            a_{j_1,i_1} & \dots & a_{j_1,i_s} \\
            & \vdots & \\
            a_{j_t,i_1} & \dots & a_{j_t,i_s}\\
            1 & \dots & 1
        \end{pmatrix}
        \in M_{t+1,s}\big(\bZ\big)
    \]
    divide $q-1$. 
    \label{asm2}
\end{assumption}

Here, note that this matrix is of rank $s$ and therefore all elementary divisors are non-zero.
In fact, since $\{i_1,\dots,i_s\}$ is a subset of $J$, the matrix above contains
\[
    \begin{pmatrix}
        a_{i_1,i_1} & \dots & a_{i_1,i_s} \\
        & \vdots & \\
        a_{i_s,i_1} & \dots & a_{i_s,i_s} \\
        1 & \dots & 1
    \end{pmatrix}
\]
as a minor matrix. Then, for each element of the kernel of the homomorphism defined by this matrix,
all the coefficients have the same sign as proved in Proposition \ref{prop:alpha},
which forces the element to be zero since all the entries of the lowest row are $1$.

\begin{example}
    (i) Consider the Dwork family $F_0(T)=T_1^{n+1}+\dots+T_{n+1}^{n+1}$;
    recall that $\alpha={}^\rt(1, 1, \dots, 1)$.
    Then, Assumptions \ref{asm1} and \ref{asm2} are equivalent to $q \equiv 1 \pmod{n+1}$.
    In fact, first, the elementary divisors of $A'$ is $1, n+1, \dots, n+1, 0$.
    Since $\Ker\big(f_{q-1}\big)$ is generated by $n$ vectors
    ${}^\rt\big((q-1)/(n+1), n(q-1)/(n+1), 0,\dots, 0\big),~{}^\rt\big(0, (q-1)/(n+1), n(q-1)/(n+1), 0,\dots,0\big),\dots,{}^\rt\big(0, \dots,0, (q-1)/(n+1), n(q-1)/(n+1)\big)$ modulo $q-1$,
    there are no extra conditions concerning this vector.
    Moreover, all elementary divisors of the matrices in Assumption \ref{asm2} above divide $n+1$.

    (ii) Consider the example from Example \ref{ex:section1_example} (ii), that is,
    $F_0(T)=T_1^4+T_1T_2^3+T_3^4+T_4^4$;
    recall that $\alpha={}^\rt(2, 4, 3, 3)$.
    Then, Assumption \ref{asm1} is equivalent to $q\equiv 1\pmod{12}$.
    Assumption \ref{asm2} moreover forces us to assume that $q\equiv 1\pmod{24}$.
    In fact, first, the elementary divisors of $A'$ is $1, 1, 4, 0$.
    Since $\Ker\big(f_{q-1}\big)$ is generated by the vector ${}^\rt\big( (q-1)/4, 0, 3(q-1)/4, 0\big)$,
    we have to impose the condition ``modulo 24'', not just ``modulo 12''.
    The $J$'s to be taken care of is $\{2, 3, 4\}, \{1, 3, 4\}, \{1, 2, 3\}, \{1, 2\},
    \{2, 3\}, \{1, 3\}, \{3, 4\}$ and these sets with $3$ and $4$ reversed.
    In fact, all elementary divisors to be considered divide 4;
    For example, the matrices in Assumption \ref{asm2} corresponding to $J=\{2,3,4\}, \{1,3,4\}, \{3,4\}$ are
    \[
        \begin{pmatrix} 0 & 0 \\ 4 & 0 \\ 0 & 4 \\ 1 & 1 \end{pmatrix},\quad
        \begin{pmatrix} 4 & 0 & 0 \\ 0 & 4 & 0 \\ 0 & 0 & 4 \\ 1 & 1 & 1\end{pmatrix},\quad
        \begin{pmatrix} 4 & 0 \\ 0 & 4 \\ 1 & 1 \end{pmatrix}
    \]
    respectively.
\end{example}

\begin{notation}
    (i) In the rest of this article, we fix a generator $\rho$ of the character group $\widehat{\bF_q^{\times}}$.

    (ii) For each natural number $\beta$ dividing $q-1$, we put $\varphi_{\beta}\defeq \rho^{(q-1)/\beta}$.
    \label{notation}
\end{notation}

\begin{notation}
    (i) For each $j=0,\dots,d-1$, we define the function $r\mapsto F(j)_r$ on $\bZ_{\geq 1}$ as follows.
    If $j=0$, then $F(0)_r$ equals
    \[
        \trghf{[\varphi_{\alpha}]}{[\varphi_{\alpha_1}],\dots,[\varphi_{\alpha_{n+1}}]}{C\lambda^{-\alpha}}_{\bF_{q^r}},
    \]
    where $[\varphi_{\beta}]$ denotes the sequence $\varepsilon, \varphi_{\beta}^1,\dots,\varphi_{\beta}^{\beta-1}$.  If $j>0$, then $F(j)_r$ equals
    \[
        q^{\delta_{|s_j|}-1} \trghf{\rho^{t_j}[\varphi_{\alpha}]}
        {\rho^{t_{1j}}[\varphi_{\alpha_1}], \dots, \rho^{t_{n+1,j}}[\varphi_{\alpha_{n+1}}]}{C\lambda^{-\alpha}}_{\bF_{q^r}},
    \]
    where $\psi[\varphi_{\beta}]$ denotes the sequence $\psi, \psi\varphi_{\beta}^1,\dots,\psi\varphi_{\beta}^{\beta-1}$, and where
    \[
        \delta_{|s_j|} \defeq \begin{cases} 1 & \text{if } |s_j|\equiv 0\pmod{q-1},\\
            0 & \text{if } |s_j|\not\equiv 0\pmod{q-1}.\end{cases}
    \]

    (ii) For each $j=0,\dots,d-1$, and $r\geq 1$, we put
    \begin{align*}
        \gamma(j)_r\defeq & \prod_{i=1}^{n+1}\rho_r^{s_{ij}}\big(\alpha_{i,r}^{-1}c_i\big)\cdot\rho_r^{s_j}\big((-\lambda)^{-1}\alpha\big) \\
        & \qquad \times \prod_{i=1}^{n+1}\left\{G\big(\rho_r^{-t_{ij}}\big)\prod_{b_i=1}^{\alpha_i-1}\frac{G\big(\rho_r^{-t_{ij}}\varphi_{\alpha_i,r}^{b_i}\big)}{G\big(\varphi_{\alpha_{i,r}}^{b_i}\big)}\right\}
        G\big(\rho_r^{t_j}\big)\prod_{b=1}^{\alpha-1}\frac{G\big(\rho_r^{t_j}\varphi_{\alpha,r}^b\big)}{G\big(\varphi_{\alpha,r}^b\big)}.
    \end{align*}
    For a positive integer $r$, the element $\gamma(j)_r$ is similarly defined
    by replacing $\rho$ by $\rho_{r}$ and $\varphi_\beta$'s by $\varphi_{\beta,r}$'s for $\beta\in\{\alpha_1,\dots,\alpha_{n+1},\alpha\}$.

    (iii) For $r\geq 1$, we put 
    \[
        u_r \defeq \sum_{\substack{J\subset\{1,\dots,n+1\}\\ \# J\geq (n+1)/2}}\sum_{i=0}^{\# J-s}(-1)^{\# J-s-i}q^{r(i-1)}\sum\prod_{j=1}^s G\big(\chi_{j,r}^{-1}\big)\chi_{j,r}(c_j);
    \]
    here, in the most inner sum, $(\chi_1,\dots,\chi_s)$ runs through $\Ker\big(\varphi(\widetilde{A})\big)$
    so that exactly $n-2i+1$ components are non-trivial.
\end{notation}

\begin{lemma}
    \label{lem:others_are_pure}
    \textup{(i)} The functions $r\mapsto \gamma(j)_r$ for $j=0,\dots,d-1$ are pure $q$-Weil numbers.
    If $j=0$, then $\gamma(0)_r=1$ for all $r>0$.
    If $j\neq 0$, then the weight of $\gamma_r(j)$ is $\#\Set{i\in\{1,2,\dots,n+1\} | s_{ij}\neq 0}+ 1 - \delta_{|s_j|}$.

    \textup{(ii)} The functions $r\mapsto \gamma(j)_rF(j)_r$ for $j=0,\dots,d-1$ is a pure $q$-Weil function of weight $n-1$.
    
    \textup{(iii)} $r\mapsto u_r$ is a $q$-Weil functions of weight $n-1$.
\end{lemma}

\begin{proof}
    (i) By the definition of $t_{ij}$'s, $\rho^{-t_{ij}}\varphi_{\alpha_i}^{b_i} ~(b_i=1,\dots,\alpha_i-1)$ cannot be
    the trivial character unless $t_{ij}\neq 0$.
    If no $\rho^{t_j}\varphi_{\alpha}^b$'s are trivial, that is, if $|s_j|\not\equiv 0\pmod{q-1}$, then
    the weight is $\#\Set{i\in\{1,2,\dots,n+1\} | s_{ij}\neq 0} + 1$ by Proposition \ref{prop:gauss_sum_is_pure}.
    If $|s_j|\equiv 0\pmod{q-1}$, then exactly one $G\big(\rho^{t_j}\varphi_{\alpha}^b\big)$ is of weight $0$,
    and the claim follows.

    (ii) If $j=0$, then the function $r\mapsto F(0)_r$ gives a pure $q$-Weil function of weight $n-1$ by Proposition \ref{prop:ghf_is_pure} and $\gamma(0)_r=1$, which shows the claim.
    Assume that $j\neq 0$. Then, the function $F(j)_r$ gives a pure $q$-Weil function
    of weight $\#\Set{i\in\{1,\dots,n+1\} | s_{ij}=0}-1-\delta_{|s_j|}+2(-1+\delta_{|s_j|})$ by Proposition \ref{prop:ghf_is_pure}.
    Therefore we see that the function $r\mapsto\gamma(j)_rF(j)_r$ is, together with (i), a pure $q$-Weil function of weight $n-1$.
    
    (iii) is a direct consequence of Proposition \ref{prop:gauss_sum_is_pure}.
\end{proof}

\subsection{Statement of the main theorem and the strategy of the proof.}

Now, we may state the main theorem.

\begin{theorem}
    Under Assumptions \ref{asm1} and \ref{asm2}, let $\lambda$ an element of $\bF_q^{\times}$ such that $X_{\lambda}$ is smooth and
    $C\neq \lambda^{\alpha}$.
    Define the polynomial $P(T)\in\bZ[T]$ by
    \[
        \zeta\big(X_{\lambda}, T\big) = \frac{P(T)^{(-1)^n}}{(1-T)(1-qT)\dots (1-q^{n-1}T)}.
    \]
    Then, $P(T)$ equals
    \[
        \zeta\big(u_r\big)(T)\big(1-q^{(n-1)/2}T\big)^D \prod_{i=0}^{d-1}\zeta\big(\gamma_r(j)F(j)_r\big)(T),
    \]
    where the number $D$ is defined to be the number of subsets $J\subset\{1,2,\dots,n+1\}$
    such that $\# J=(n+1)/2$ and that for all $i=1,\dots,n+1$ there exists $j\not\in I$ such that $a_{ij}\geq 1$
    \textup{(}in particular, it is zero if $n$ is even\textup{)}.
    \label{thm:factorization_zeta}
\end{theorem}

\begin{remark}
    (i) Note that the parameters of the hypergeometric function $F(0)_r$ corresponds to
    the parameters of $\sF(x)$ in the $p$-adic situation (\ref{eq:padic_ghf}).

    (ii) Lemma \ref{lem:others_are_pure} (ii) shows that each $\zeta\big(\gamma(j)_rF(j)_r\big)(T)$ is a polynomial.
    Moreover, this is the Frobenius trace of an \emph{irreducible} $\overline{\bQ_{\ell}}$-sheaf on $\Gm$ at a point
    \cite[Theorem 8.4.2]{Katz:ESDE}, as we explained in Subsection 1.1.

    (iii) $\zeta(u_r)(T)$ is a rational function all of whose reciprocal zeros and poles
    belong to $\bQ(\mu_{q-1})$, where $\mu_{q-1}$ is the subgroup of $\overline{\bQ}^{\times}$ of order $q-1$.
\end{remark}

In order to prove this theorem, it clearly suffices to show the following proposition.

\begin{proposition}
    The number of $\bF_{q^r}$-rational points of $X_{\lambda}$ equals
    \[
        \sum_{i=0}^{n-1}(q^r)^i + u_r + D(q^r)^{(n-1)/2} + (-1)^n\sum_{j=0}^{d-1}\gamma(j)_rF(j)_r,
    \]
    where $F(j)_r$ is the function defined in the statement of the theorem. 
    \label{prop:count_points}
\end{proposition}

The key point of the proof of this proposition is to calculate the number of $\bF_{q^r}$-rational points
of $X_{\lambda}$ ``modulo $q$-Weil function of weight $\neq n-1$'';
in fact, we know from Weil conjecture that the ``weight $\neq n-1$ part'' of this sum equals $1+q+\dots+q^{n-1}$.
In the next subsection, we prove the following assertion in the next subsection,
from which Proposition \ref{prop:count_points} directly follows.

\begin{proposition}
    The function $r\mapsto X_{\lambda}\big(\bF_{q^r}\big)$ is the sum of the function
    \[ 
        r \mapsto \sum_{i=0}^{n-1}(q^r)^i + u_r + D (q^r)^{(n-1)/2} + \sum_{j=0}^{d-1}\gamma(j)_rF(j)_r
    \]
    and a $q$-Weil function of weight $\neq n-1$.
    \label{prop:count_points_via_function}
\end{proposition}

\subsection{Proof: Counting rational points.}

In this subsection, we fix an element $\lambda$ of $\bF_q^{\times}$ such that $X_{\lambda}$ is smooth
and $C\neq\lambda^{\alpha}$, and prove Proposition \ref{prop:count_points_via_function}.

First, we recall a classical formula that expresses the number of rational points
of arbitrary hypersurface of $\bG_m$ in terms of Gauss sums.

\begin{notation}
Let $n$ and $N$ be positive integers and let $M=(m_{ij})_{i,j}$ be an $n\times N$ matrix with coefficients in $\bZ$.
We denote by $\varphi(M)$ the natural homomorphism
$\big(\widehat{\bF_q^{\times}}\big)^N\to \big(\widehat{\bF_q^{\times}}\big)^n$ defined by $M$, which is
explicitly expressed as
\[
    \varphi(M) \big((\chi_i)_{i=1,\dots,N}\big) = \left( \chi_1^{m_{j1}}\dots\chi_n^{m_{jN}}\right)_{j=1,\dots,n}.
\]
We always regard elements of $\big(\widehat{\bF_q^{\times}}\big)^N$ and $\big(\widehat{\bF_q^{\times}}\big)^n$ as column vectors.
\end{notation}

The following general theorem is classical \cite{Delsarte}, \cite{Gomida}; the readers also can 
find the proof in this context in \cite[Proposition 3]{MiyataniSano}.

\begin{proposition}
    \label{prop:hypersurfaceofgm}
    Let $n$ and $N$ be positive integers, let $c_1,\dots,c_N$ be elements of $\bF_q^{\times}$, and
    let $R=(r_{ij})_{i,j}\in M_{n,N}(\bZ)$ be an $n\times N$ matrix with coefficients in $\bZ$.
    Define a polynomial $f(X_1,\dots,X_n)\in\bF_q[X_1,\dots,X_n]$ by
    \[
        f(X_1,\dots,X_n)=\sum_{j=1}^Nc_jX_1^{r_{1j}}\dots X_n^{r_{nj}}.
    \]
    Then, the number of $n$-tuples $(x_1,\dots,x_n)\in(\bF_q^{\times})^n$ satisfying $f(x_1,\dots,x_n)=0$ equals
    \[
        \frac{(q-1)^n}{q} + \frac{(q-1)^{n+1-N}}{q}\sum_{(\chi_1,\dots,\chi_N)\in \Ker(\varphi(\widetilde{R}))}\prod_{j=1}^NG\big(\chi_j^{-1}\big)\chi_j(c_j),
    \]
    where
    \[
        \widetilde{R} = \left(\begin{array}{ccc} \multicolumn{3}{c}{\raisebox{-4pt}{\large $R$}}\\ && \\ \hline 1 & \dots & 1\end{array}\right).
    \]
\end{proposition}

Now, let $X_{\lambda}\big(\bF_q\big)_0$ be the set of $\bF_q$-rational points $[x_1:\dots:x_{n+1}]$ of
$X_{\lambda}$ at least one of whose coordinates is zero, and
set $X_{\lambda}\big(\bF_q\big)_\ast := X_{\lambda}\big(\bF_q\big) \setminus X_{\lambda}\big(\bF_q\big)_0$.

\begin{proposition}
    The function $r\mapsto \# X_{\lambda}\big(\bF_{q^r}\big)_0$ is the sum of the function
    \[
        r \mapsto \sum_{i=0}^{n-1} (q^r)^i + D (q^r)^{(n-1)/2}  -\frac{(q^r-1)^n}{q^r} +u_r
    \]
    and a $q$-Weil function of weight $\neq n-1$.
    \label{prop:count_zero}
\end{proposition}

\begin{proof}
    For a proper non-empty subset $J\subset\{1,\dots,n+1\}$ and a positive integer $r$,
    denote by $s_r(J)$ the number of $\bF_{q^r}$-rational points $[x_1:\dots:x_{n+1}]$ of
    $X_{\lambda}$ such that $x_j\neq 0$ if and only if $j\in J$.
    Then, we have
    \[
        \# X_{\lambda}\big(\bF_{q^r}\big)_0 = \sum_{t=1}^{n}\sum_{\# J=t}s_r(J).
    \]
    By Lemma \ref{lem:invertible_on_J} below, each $s_r(J)$ is of the form
    \[
        \frac{(q^r-1)^{t-1}}{q^r} + N'_J(q^r)^{(n-1)/2}+
        \sum_{i=0}^{\# J-s}(-1)^{\# J-s-i}q^{i-1}\sum\prod_{j=1}^sG\big(\chi_j^{-1}\big)\chi_j(c_j)
    \]
    plus a $q$-Weil function of weight $\neq n-1$,
    where $N'_J$ is the number given in the statement of Lemma \ref{lem:invertible_on_J}.
    This shows that $\# X_{\lambda}\big(\bF_{q^r}\big)_0$ is written as
    \begin{equation}
        \sum_{t=1}^{n}\binom{n+1}{t}\frac{(q^r-1)^{t-1}}{q^r} + D(q^r)^{(n+1)/2} + u_r
        \label{eq:tmp_sum}
    \end{equation}
        plus a $q$-Weil function of weight $\neq n-1$.
    Since
    \[
        \sum_{t=0}^{n+1}\binom{n+1}{t}\frac{(q^r-1)^t}{q^r}=(q^r)^n,
    \]
    the first term of (\ref{eq:tmp_sum}) equals
    \[
        \frac{(q^r)^n}{q^r-1} - \frac{1}{q^r(q^r-1)}-\frac{(q^r-1)^n}{q^r}
        = \frac{(q^r)^n-1}{q^r-1}-\frac{(q^r-1)^n}{q^r} + \frac{1}{q^r}.
    \]
    
%    Now, it suffices to show the equality
%    \begin{equation}
%        \sum_{i=1}^{n-1}(-1)^{i-1}\binom{n+1}{i}\sum_{j=0}^{n-i-1} q^j = \sum_{i=0}^{n-1}q^i - \frac{(q-1)^n - (-1)^n}{q}.
%        \label{eq:binomial_relation}
%    \end{equation}
%    The left-hand side of (\ref{eq:binomial_relation}) equals
%    \begin{equation}
%        \sum_{i=1}^{n-1}(-1)^{i-1}\binom{n+1}{i}\frac{q^{n-i}-1}{q-1} = -\frac{1}{q-1}\sum_{i=1}^{n-1}(-1)^i\binom{n+1}{i}q^{n-i} + \frac{1}{q-1}\sum_{i=1}^{n-1}(-1)^i\binom{n+1}{i}.
%        \label{eq:binomial_relation2}
%    \end{equation}
%    Since we have, not only for our $q$ but also for all positive integer $q$
%    \begin{eqnarray*}
%        -\sum_{i=1}^{n-1}(-1)^i\binom{n+1}{i}q^{n-i} &=& -\frac{1}{q}\sum_{i=0}^{n+1}(-1)^i\binom{n+1}{i}q^{n-i+1} + q^n + (-1)^n(n+1) + (-1)^{n+1}\frac{1}{q} \\
%            &=& -\frac{(q-1)^{n+1}}{q} + q^n + (-1)^n(n+1) + \frac{(-1)^{n+1}}{q},
%    \end{eqnarray*}
%    and since in particular we have
%    \[
%        \sum_{i=1}^{n-1}(-1)^i\binom{n+1}{i}= -1-(-1)^n(n+1)-(-1)^{n+1},
%    \]
%    the right-hand side of (\ref{eq:binomial_relation2}) equals
%    \[
%        -\frac{(q-1)^n}{q} + \frac{q^n-1}{q-1}+\frac{(-1)^{n+1}}{q-1}\left(\frac{1}{q}-1\right) = \sum_{i=0}^{n-1}q^i - \frac{(q-1)^n-(-1)^n}{q},
%    \]
%    which shows the equation (\ref{eq:binomial_relation}).
\end{proof}

\begin{lemma}
    Let $J$ be a subset of $\{1,\dots, n+1\}$ that satisfies $1\leq \# J\leq n$.
    We define a number $N'_J$ as follows;
    it is $1$ if $\# J=(n+1)/2$ and for each $i\in\{1,2,\dots,n+1\}$ there exists $j\not\in J$ with
    $a_{ji}\geq 1$; it is $0$ otherwise.
    Then, the function
    \[
        r \mapsto s_r(J) \defeq \#\Set{[x_1:\dots:x_{n+1}]\in X_{\lambda}\big(\bF_{q^r}\big) | x_j\neq 0 \text{ if and only if } j\in J}.
    \]
    can be written as the sum of the function
    \[
        \frac{(q^r-1)^{\# J-1}}{q^r} + N'_J(q^r)^{(n-1)/2} + 
        \sum_{i=0}^{\# J-s}(-1)^{\# J-s-i}q^{i-1}\sum\prod_{j=1}^sG\big(\chi_j^{-1}\big)\chi_j(c_j)
    \]
    and a $q$-Weil function of weight $\neq n-1$.
    Here, $(\chi_1,\dots,\chi_s)$ runs through $\Ker(\varphi(\widetilde{A}))$
    so that exactly $n-2i+1$ components are non-trivial.
    \label{lem:invertible_on_J}
\end{lemma}

\begin{proof}
    If $\# J\leq n/2$, then $s_r(J)$ can be considered as the number of $\bF_{q^r}$-rational points
    of a closed subscheme of $\Gm_{,\bF_q}^{\# J-1}$ with $\# J-1\leq n/2-1$,
    and therefore each term of the function in the statement is a $q$-Weil function of weight $\leq n-2$.
    Therefore, we assume that $\# J\geq(n+1)/2$.
    %Assume that $F_0(T)$ equals zero after substituting $T_{i_1}=\dots=T_{i_t}=0$;
    %by Proposition \ref{prop:form_of_A}, this forces $n$ to be odd and $t$ to equal $(n+1)/2$.
    %Therefore, we may assume that the polynomial $F_0(T)$ does not vanish after substituting
    %$T_{i_1}=\dots=T_{i_t}=0$.
    %In this case, $X_I$ may be considered as a hypersurface of $\bP^{n-t}_{\bF_q}$.
    %If $X_I$ is smooth, the function
    %\[
    %    r \mapsto (-1)^{n-t-1}\left(X_I\big(\bF_{q^r}\big) - \sum_{i=1}^{n-t-1}(q^r)^j\right)
    %\]
    %is a pure $q$-Weil function of weight $n-t-1$.
    %We, therefore, assume that $X_I$ is singular; in this case, there are some $i\in I$ and $i'\not\in I$
    %such that $a_{i,i'}=1$ since otherwise $X_I$ is isomorphic to a Fermat hypersurface of $\bP^{n-t}_{\bF_q}$.
    %First, we fix a subset $J$ of $\{1,\dots, n+1\}\setminus I$ that satisfies $t_J\defeq\sharp J\geq 1$,
    %and we show that the function
    %\[
    %    r \mapsto \sS_J(r) \defeq \sharp\Set{[x_1:\dots:x_{n-t+1}]\in X_I\big(\bF_{q^r}\big) | x_j\neq 0 \text{ if and only if } j\in J}.
    %\]
    %can be written as
    %\[
    %    \frac{(q^r-1)^{t_J-1}}{q^r} + N'_Jq^{(n-1)/2} + 
    %    \big(\text{a $q$-Weil function of weight $\leq n-2$}\big),
    %\]
    %where $N'_J$ is $1$ if $\sharp J=(n+1)/2$ and for each $i\in\{1,2,\dots,n+1\}$ there exists $j\in J$ with
    %$a_{ji}\geq 1$; it is defined to be zero otherwise.
    %Again, we may assume that $t_J\geq (n+1)/2$.
    If $N'_J=1$, then the claim is also trivial since $X_J$ is isomorphic to $\Gm_{\bF_q}^{(n+1)/2-1}$;
    now we assume that $N'_J=0$.
    For simplicity, we change the coordinates so that $J=\{1,\dots,\# J\}$,
    and so that the following two conditions hold for an $s\in\{1,2,\dots,\# J\}$:
    
    (i) $a_{ij}=0$ for all $i\not\in J$ and $j=1,\dots,s$ and

    (ii) for all $j=s+1,\dots,n+1$, there exists an $i\not\in J$ such that $a_{ij}\geq 1$.

    Now, by Proposition \ref{prop:hypersurfaceofgm},
    \begin{eqnarray}
        s_1(J) &=& \frac{1}{q-1}\#\Set{ (x_1,\dots,x_{\# J})\in\bF_q^{\times} | c_1x^{a_1}+\dots+c_sx^{a_s}=0 }\nonumber\\
        &=& \frac{(q-1)^{\# J-1}}{q} + \frac{(q-1)^{\# J-s}}{q}\sum_{(\chi_1,\dots,\chi_s)\in\Ker(\varphi(\widetilde{A_J}))}\prod_{j=1}^sG\big(\chi_j^{-1}\big)\chi_j(c_j),
        \label{eq:toric}
    \end{eqnarray}
    where $A_J$ denotes the matrix
    \[
        \begin{pmatrix}
            a_{11} & \dots & a_{1s} \\
            & \vdots & \\
            a_{\# J,1} & \dots & a_{\# J,s}\\
        \end{pmatrix}.
    \]
    Therefore, if we denote by $d_1',\dots,d_s'$ the elementary divisors of the $(\# J+1)\times s$ matrix $\widetilde{A_J}$,
    then by hypothesis, we see as in the argument after Assumption \ref{asm2} that
    the second term of (\ref{eq:toric}) is $(q-1)$ times the sum of $d_1'\dots d_s'$ numbers
    that are the products of $s\leq t$ Gauss sums.

    We may do this computation in the same way for various $r$, 
    and therefore, the statement similar to the argument after Assumption \ref{asm2}
    and Proposition \ref{prop:gauss_sum_is_pure} shows the claim.
    
%    Now, by summing up $\sS_J(r)$'s for all $J\subset\{1,2,\dots,n+1\}\setminus I$, we get
%    \[
%        s_r(I) = \frac{1}{q^r}\sum_{t'=1}^{n-t+1}\binom{n-t+1}{t'}(q^r-1)^{t'-1} + N_I(q^r)^{(n-1)/2}
%        + \big(\text{a $q$-Weil function of weight $\leq n-2$}\big),
%    \]
%    and the right-hand side may be written as
%    \[
%        \frac{1}{q^r}\left\{\sum_{t'=0}^{n-t+1}\binom{n-t+1}{t'}(q^r-1)^{t'-1}-(q^r-1)^{-1}\right\} + N_I(q^r)^{(n-1)/2} +
%        \big(\text{a $q$-Weil function of weight $\leq n-2$}\big).
%    \]
%    The first term equals $\big((q^r)^{n-t}-1\big)/\big(q^r-1\big)$, which can be proven as in the proof of
%    the previous subsection.
\end{proof}

Now, let us investigate the function $r\mapsto \# X_{\lambda}\big(\bF_{q^r}\big)_{\ast}$.
For the investigation, we prove the following fact.

\begin{proposition}
    The function $r\mapsto \# X_{\lambda}\big(\bF_{q^r}\big)_\ast$ is the sum of the function
    \[
        r\mapsto \frac{(q^r-1)^n}{q^r}+(-1)^n\sum_{i=0}^{d-1}\gamma(j)_rF(j)_r 
    \]
    and a $q$-Weil function of weight $\leq n-2$.
    \label{prop:count_star}
\end{proposition}

\begin{proof}
    In this proof, we only investigate the number $X_{\lambda}\big(\bF_{q^r}\big)_{\ast}$ for $r=1$
    because we may compute it for general $r$ in exactly the same way.

    By using Proposition \ref{prop:hypersurfaceofgm}, $\# X_{\lambda}\big(\bF_q\big)_{\ast}$ equals
    \begin{equation}
        \frac{(q-1)^{n}}{q} + \frac{1}{q(q-1)}\sum_{ {}^{\rt}(\chi_1,\dots,\chi_{n+2})\in\Ker\varphi(\widetilde{A})}
        \prod_{i=1}^{n+1}G\big(\overline{\chi_i}\big)\chi_i(c_i)\cdot G\big(\overline{\chi_{n+2}}\big)\chi_{n+2}(-\lambda).
        \label{eq:count_star}
    \end{equation}
%    Now, our task is the calculation of the second term.
%
%    By Proposition \ref{prop:hypersurfaceofgm}, it equals $1/(q-1)^{n+2}$ times
%    \[
%        \sum_{\chi_1,\dots,\chi_{n+1},\mu\in\widehat{\bF_q^{\times}}}\prod_{i=1}^{n+1}G\big(\overline{\chi_i}\big)\cdot G\big(\overline{\mu}\big)\prod_{i=1}^{n+1}\chi_i\big(c_i\big)
%        \sum_{w\in\bF_q^{\times}}\chi_1\dots\chi_{n+1}\mu\big(w\big)\prod_{i=1}^{n+1}\sum_{x_i\in\bF_q^{\times}}
%        \chi_1^{a_{i1}}\dots\chi_{n+1}^{a_{i,n+1}}\mu\big(x_i\big)\cdot \mu\big(-\lambda\big).
%    \]
    Let us describe the kernel of $\varphi(\widetilde{A})$.
    Recall that we fixed a generator $\rho$ of $\widehat{\bF_q^{\times}}$ in Notation \ref{notation},
    and take $k_i\in\bZ$ for $i=1,\dots,n+2$ so that $\chi_i=\rho^{k_i}$.
    With this notation, the definition of $s_{ij}$'s shows that
    ${}^{\rt}(\chi_1,\dots,\chi_{n+2})$ is an element of $\Ker(\varphi(\widetilde{A}))$
    if and only if there exists an index $j=0,\dots,d-1$ and $a=0,\dots,q-2$ such that
    \[
        k_i \equiv s_{ij} + a\alpha_i\quad (i=1,\dots,n+1)\quad\text{and}\quad k_{n+2} \equiv -|s_j|-a\alpha;
    \]
    in this case, the choice of $j$ and $a$ is unique.
    Now, we have shown that the second term of (\ref{eq:count_star}) is
    \begin{eqnarray*}
        && \sum_{j=0}^{d-1}\sum_{a=0}^{q-2}\left\{\prod_{i=1}^{n+1}G\big(\rho^{-s_{ij}-a\alpha_i}\big)\cdot
        G\big(\rho^{|s_j|+a\alpha}\big) \prod_{i=1}^{n+1}\rho^{s_{ij}+a\alpha_i}\big(c_i\big)\rho^{-|s_j|-a\alpha}\big(-\lambda\big)\right\}\\
        &=& \sum_{j=0}^{d-1}\left\{\prod_{i=1}^{n+1}\rho^{s_{ij}}\big(c_i\big)\cdot\rho^{-|s_j|}\big(-\lambda\big)
        \sum_{a=0}^{q-2}\prod_{i=1}^{n+1}G\big(\rho^{-s_{ij}-a\alpha_i}\big)\cdot
        G\big(\rho^{|s_j|+a\alpha}\big) \rho^{a}\left((-1)^{\alpha}\prod_{i=1}^{n+1}c_i^{\alpha_i}\cdot\lambda^{-\alpha}\right)\right\}.
    \end{eqnarray*}
    For each $j$, we have $s_{ij}=\alpha_{i}t_{ij}$ and $|s_j|=\alpha t_j$, therefore 
    Lemma \ref{lem:multi_hyper} below shows that
    \begin{eqnarray*}
        && \frac{1}{q-1}\sum_{a=0}^{q-2}\prod_{i=1}^{n+1}G\big(\rho^{-s_{ij}-a\alpha_i}\big)\cdot
        G\big(\rho^{|s_j|+a\alpha}\big) \rho^a\left((-1)^{\alpha}\prod_{i=1}^{n+1}c_i^{\alpha_i}\cdot\lambda^{-\alpha}\right)\\
        &=& (-1)^n\prod_{i=1}^{n+1}\rho^{-t_{ij}}\big(\alpha_i^{-\alpha_i}\big)\cdot\rho^{t_j}\left(\alpha^{\alpha}\right)
        \prod_{i=1}^{n+1}\left\{G\big(\rho^{-t_{ij}}\big)\prod_{b_i=1}^{\alpha_i-1}\frac{G\big(\rho^{-t_{ij}}\varphi_{\alpha_i}^{b_i}\big)}{G\big(\varphi_{\alpha_i}^{b_i}\big)}\right\}
        G\big(\rho^{t_j}\big)\prod_{b=1}^{\alpha-1}
        \frac{G\big(\rho^{t_j}\varphi_{\alpha}^b\big)}{G\big(\varphi_{\alpha}^b\big)}\\
        && \qquad \times\tghf{\alpha}{\alpha}{\rho^{t_j}[\varphi_{\alpha}]}
        {\rho^{t_{1j}}[\varphi_{\alpha_1}],\dots, \rho^{t_{n+1,j}}[\varphi_{\alpha_{n+1}}]}{C\lambda^{-\alpha}}_{\bF_q}.
    \end{eqnarray*}

    Finally, the second term of (\ref{eq:count_star}) for various $r$
    is the sum of the function $r\mapsto (-1)^n\sum_{j=0}^{d-1}\gamma(j)_rF(j)_r$
    and a $q$-Weil function of weight $\leq n-2$ by Proposition \ref{prop:cancellation};
    use (i) for $j=0$ and for $j\neq 0$ such that $\delta_{|s_j|}=1$, and use (ii) for the other $j$'s.
\end{proof}

\begin{lemma}
    Let $\alpha_1,\dots,\alpha_{n+1}$ be positive integers, put $\alpha:=\alpha_1+\dots+\alpha_{n+1}$, and
    assume that $q$ is congruent to $1$ modulo all $\alpha_i$'s and modulo $\alpha$.

    Let $A_1,\dots,A_{n+1},B$ be characters on $\bF_q^{\times}$.
    Then, we have the equation
    \begin{eqnarray*}
        && \sum_{\chi\in\widehat{\bF_q^{\times}}}\prod_{i=1}^{n+1}G\big((\overline{A_i\chi})^{\alpha_i}\big)
        G\big((B\chi)^{\alpha}\big)\chi\big((-1)^{\alpha}x\big)\\
        &=& (-1)^n(q-1)\prod_{i=1}^{n+1}A_i\big(\alpha_i^{-\alpha_i}\big)\cdot B\big(\alpha^{\alpha}\big)
        \prod_{i=1}^{n+1}\left\{G\big(\overline{A_i}\big)\prod_{b_i=1}^{\alpha_i-1}\frac{G\big(\overline{A_i}\varphi_{\alpha_i}^{b_i}\big)}{G\big(\varphi_{\alpha_i}^{s_i}\big)}\right\}\cdot G\big(B\big)
        \prod_{b=1}^{\alpha-1}\frac{G\big(B\varphi_{\alpha}^b\big)}{G\big(\varphi_{\alpha}^b\big)}\\
        && \qquad \times \tghf{\alpha}{\alpha}{B[\varphi_{\alpha}]}{A_1[\varphi_{\alpha_1}],\dots,A_{n+1}[\varphi_{\alpha_{n+1}}]}{\frac{\alpha^{\alpha}}{\alpha_1^{\alpha_1}\dots\alpha_{n+1}^{\alpha_{n+1}}}x}_{\bF_q}.
    \end{eqnarray*}
    \label{lem:multi_hyper}
\end{lemma}

\begin{proof}
    Davenport--Hasse relation \cite[$(0.9_1)$]{Davenport-Hasse:NKGZF} shows that
    \begin{eqnarray*}
        G\big((\overline{A_i\chi})^{\alpha_i}\big) &=& -G\big(\overline{A_i\chi}\big)\prod_{s_i=1}^{\alpha_i-1}
        \frac{G\big(\overline{A_i\chi}\varphi_{\alpha_i}^{s_i}\big)}{G\big(\varphi_{\alpha_i}^{s_i}\big)}\big(A_i\chi\big)\big(\alpha_i^{-\alpha_i}\big),\\
        G\big((B\chi)^{\alpha}\big) &=& -G\big(B\chi\big)\prod_{s=1}^{\alpha-1}
        \frac{G\big(B\chi\varphi_{\alpha}^s\big)}{G\big(\varphi_{\alpha}^s\big)}\big(B\chi\big)\big(\alpha^{\alpha}\big).
    \end{eqnarray*}
    Therefore, the left-hand side of the equation in the statement equals
    \begin{eqnarray*}
        &&(-1)^n\prod_{i=1}^{n+1}A_i\big(\alpha_i^{-\alpha_i}\big)\cdot B\big(\alpha^{\alpha}\big)\\
        && \qquad \times
        \sum_{\chi\in\widehat{\bF_q^{\times}}}\prod_{i=1}^{n+1} G\big(\overline{A_i\chi}\big)\cdot G\big(B\chi\big)
        \prod_{i=1}^{n+1}\prod_{b_i=1}^{\alpha_i-1}
        \frac{G\big(\overline{A_i\chi}\varphi_{\alpha_i}^{b_i}\big)}{G\big(\varphi_{\alpha_i}^{b_i}\big)}\cdot
        \prod_{b=1}^{\alpha-1}
        \frac{G\big(B\chi\varphi_{\alpha}^b\big)}{G\big(\varphi_{\alpha}^b\big)}
        \chi\left((-1)^{\alpha}\frac{\alpha^{\alpha}}{\alpha_1^{\alpha_1}\dots\alpha_{n+1}^{\alpha_{n+1}}}x\right).
    \end{eqnarray*}
    This shows the proposition.
\end{proof}

Proposition \ref{prop:count_zero} and \ref{prop:count_star} completes the proof of Proposition \ref{prop:count_points_via_function}.

\section*{Acknowledgements.}

This article is based on the doctoral thesis of the author.
The author would like to express his sincere gratitude to his advisor Atsushi Shiho for
having meaningful discussions with him, reading carefully the draft of this paper,
pointing out a lot of mistakes on it, giving him valuable suggestions and
encouraging him patiently and warmly.

This work was supported by Grant-in-Aid for JSPS Fellows (24-9501).

\bibliographystyle{dagaz}
\bibliography{math}

\end{document}